\documentclass[11pt,a4paper,leqno]{amsart}

\title[Lagrangian distributions and Shubin FIOs]
{Lagrangian distributions and Fourier integral operators with quadratic phase functions and Shubin amplitudes}

\author[M. Cappiello]{Marco Cappiello}

\address{Department of Mathematics, University of Torino, Via Carlo Alberto 10, 10123 Torino, Italy.}

\email{marco.cappiello[AT]unito.it}

\author[R. Schulz]{Ren\'e Schulz}

\address{Leibniz Universit\"at Hannover, Institut f\"ur Analysis, Welfenplatz 1, D--30167 Hannover, Germany}

\email{rschulz[AT]math.uni-hannover.de}

\author[P. Wahlberg]{Patrik Wahlberg}

\address{Department of Mathematics, Linn{\ae}us University, SE--351 95 V\"axj\"o, Sweden}

\email{patrik.wahlberg[AT]lnu.se}

\usepackage{tikz,xifthen}

\usepackage{latexsym}
\usepackage[a4paper]{geometry}
\usepackage{amsmath}
\usepackage{amssymb}
\usepackage{amsthm}
\usepackage{bbm}
\usepackage{amsfonts}
\usepackage{mathrsfs}
\usepackage{calc}
\usepackage{cite}
\usepackage{color}
\usepackage{graphicx}
\usepackage{enumerate}

\numberwithin{equation}{section}          

\newtheorem{thm}{Theorem}
\numberwithin{thm}{section}

\newcommand{\rubrik}{}
\newtheorem{prop}[thm]{Proposition}
\newtheorem{cor}[thm]{Corollary}
\newtheorem{lem}[thm]{Lemma}

\theoremstyle{definition}

\newtheorem{defn}[thm]{Definition}
\newtheorem{example}[thm]{Example}

\theoremstyle{remark}

\newtheorem{rem}[thm]{Remark}              

\newcommand{\Ran}{\operatorname{Ran}}
\newcommand{\Ker}{\operatorname{Ker}}

\newcommand{\pdd}[2] {\partial_{#1} ^{#2}}

\newcommand{\ro}{\mathbb R}
\newcommand{\no}{\mathbb N}
\newcommand{\rr}[1]{\mathbb R^{#1}}
\newcommand{\nn}[1]{\mathbb N^{#1}}

\newcommand{\co}{\mathbb C}

\newcommand{\dd}{\mathrm {d}}

\newcommand{\ep}{\varepsilon}
\newcommand{\fy}{\varphi}

\newcommand{\supp}{\operatorname{supp}}

\newcommand{\wpr}{{\text{\footnotesize $\#$}}}

\newcommand{\eabs}[1]{\langle #1\rangle}

\newcommand{\ran}{\operatorname{Ran}}
\newcommand{\Sp}{\operatorname{Sp}}
\newcommand{\GL}{\operatorname{GL}}
\newcommand{\M}{\operatorname{M}}
\newcommand{\On}{\operatorname{O}}
\newcommand{\OP}{\operatorname{OP}}
\newcommand{\Mp}{\operatorname{Mp}}

\newcommand{\cI}{\mathscr{I}}

\newcommand{\dbar}{{{{\ \mathchar'26\mkern-12mu \mathrm d}}}}

\newcommand{\WF}{\mathrm{WF}}

\newcommand{\diag}{\mathrm{diag}}
\newcommand{\dist}{\operatorname{dist}}

\newcommand{\cS}{\mathscr{S}}
\newcommand{\cT}{\mathcal{T}}
\newcommand{\cV}{\mathcal{V}}
\newcommand{\cTp}{\mathcal{T}_{\psi_0}}

\newcommand{\cF}{\mathscr{F}}
\newcommand{\cO}{\mathcal{O}}
\newcommand{\cK}{\mathscr{K}}

\newcommand{\cR}{\mathscr{R}}

\newcommand{\J}{\mathcal{J}}

\def\la{\langle}
\def\ra{\rangle}
\newcommand{\leqs}{\leqslant}
\newcommand{\geqs}{\geqslant}

\begin{document}

\begin{abstract}
We study Fourier integral operators with Shubin amplitudes and quadratic phase functions associated to twisted graph Lagrangians with respect to symplectic matrices.  
We factorize such an operator as the composition of a Weyl pseudodifferential operator and a metaplectic operator and derive 
a characterization of its Schwartz kernel in terms of phase space estimates. 
Extending the conormal distributions in the Shubin calculus, we define an adapted notion of Lagrangian tempered distribution. 
We show that the kernels of Fourier integral operators are identical to Lagrangian distributions with respect to twisted graph Lagrangians. 
\end{abstract}


\maketitle

\section{Introduction}

Lagrangian distributions were introduced by H\"ormander \cite[Vol.~IV]{Hormander0} as a framework for a global theory of Fourier integral operators (FIOs). 
FIOs are defined as operators whose Schwartz kernel is a Lagrangian distribution associated to a canonical relation. 
Many of the properties of FIOs can be deduced from the study of their kernels. 
A special case of Lagrangian distributions are the conormal distributions, cf. \cite[Vol.~III]{Hormander0}, which include the kernels of pseudodifferential operators. 
Lagrangian and conormal distributions are defined in terms of local Besov norm estimates which are required to be preserved under the action of certain pseudodifferential operators. 
These estimates reflect properties of the amplitudes of the operators, which are in the classical setting the H\"ormander symbols. 

In this paper we consider another fundamental class of operators in the theory of partial differential equations, the so called Shubin class \cite{Shubin1}. 
A Shubin symbol $a \in \Gamma^m_\rho (\rr {2d})$, with $m \in \ro$ and $0 \leqs \rho \leqs 1$, satisfies the estimates
\begin{equation*}
\left| \partial_x^\alpha \partial_\xi^\beta a (x,\xi) \right| \lesssim (1+|x|+|\xi|)^{m-\rho|\alpha+\beta|}, \qquad (x,\xi) \in T^* \rr d, \alpha,\beta \in \nn d. 
\end{equation*}
Shubin symbols for pseudodifferential operators are interesting not least since they encompass the symbol $a(x,\xi) = |x|^2+|\xi|^2$ of the harmonic oscillator. 
The Shubin symbols behave isotropically on the phase space $T^* \rr d$, 
which distinguishes them from the H\"ormander symbols. 
The Shubin class has been intensively studied by many authors, e.g. \cite{Asada1, BBR, CGR, CN, CRT, Helffer1, HR1, Hormander1, Nicola1, Rodino1, PRW1, SW2, Shubin1}.  

Concerning FIOs with Shubin amplitudes, the main contributions include Asada and Fujiwara \cite{Asada1} and Helffer and Robert \cite{Helffer1, HR1}, who developed the calculus and the spectral theory, and gave applications to PDEs.
They used real phase functions that are generalizations of quadratic forms, with a prescribed condition of non-degeneracy. 
Such phase functions deviate from H\"ormander FIO phase functions that are homogeneous of degree one in the covariable. 
More recently Cordero et al. \cite{Cordero2,CGNR, CNR} and Tataru \cite{Tataru} have contributed to the field of FIOs with quadratic type phase functions, in the former case using amplitudes from modulation spaces rather than Shubin type amplitudes. 

A theory of conormal and Lagrangian distributions for FIOs with Shubin amplitudes and quadratic phase functions, parallel to H\"orman\-der's theory \cite[Vol.~IV]{Hormander0}, and reflecting the peculiar properties of the kernels of these operators, is still missing in the literature. 

In \cite{Cappiello2} we started to fill this gap by defining the space of $\Gamma$-conormal tempered distributions, adapted to Shubin pseudodifferential operators. 
The definition concerns estimates of certain differential operators acting on a Fourier--Bros--Iagolnitzer (FBI) transform of the distribution. 
Inspired by \cite[Chapter~18.2]{Hormander0} we showed that the kernels of Shubin pseudodifferential operators are exactly $\Gamma$-conormal distributions with respect to the diagonal subspace in $\rr d \times \rr d$, cf. \cite[Example~5.2]{Cappiello2}. 

In the present paper we extend the $\Gamma$-conormal to $\Gamma$-Lagrangian distributions and corresponding FIOs. 
A main result extends \cite[Example~5.2]{Cappiello2} from pseudodifferential operators to FIOs and reads as follows. 
The space of kernels of FIOs with Shubin amplitude and quadratic phase function associated with a twisted graph Lagrangian with respect to a symplectic matrix, is identical to the space of $\Gamma$-Lagrangian distributions on $\rr {2d}$ with respect to the same twisted graph Lagrangian. 

Conceptually this result is a Shubin version of a fundamental result for classical FIOs with phase function that is homogeneous  of degree one in the covariable, cf. \cite[Chapter~25]{Hormander0}. 
Indeed FIOs are in \cite[Definition~25.2.1]{Hormander0} defined as operators whose kernel are Lagrangian 
with respect to a twisted canonical relation. 
Lagrangian distributions are shown to be locally representable as oscillatory integrals, and vice versa. 
Our approach is the opposite: We define Shubin FIOs using oscillatory integrals and then we prove that their kernels are $\Gamma$-Lagrangian distributions. 
But the conclusion that the kernels of FIOs are exactly Lagrangian distributions remains the same. 

The key tools in our approach are an FBI type transform, used already in \cite{Cappiello2}, and metaplectic operators. 
The idea to study estimates in the phase space of an FBI transform is suggested by the isotropic behavior of the amplitudes, 
and by analogous estimates proved for similar operators, cf. \cite{Tataru}. 
This approach leads us to restrict to quadratic phase functions whose associated Lagrangian is a twisted graph Lagrangian in $T^* \rr {2d}$ with respect to a symplectic matrix. 
Under this restriction the calculus of the FIOs turns out to be contained in the analysis in \cite{Asada1, Helffer1}. 
However, we refine those calculi with respect to behavior under composition. 
In our composition result Proposition \ref{prop:composition} we obtain homomorphism with respect to the symplectic matrices associated to the phase functions. 

This feature turns out to have many consequences. 
The most important consequence is the factorization result Theorem \ref{thm:repFIO} which says that an FIO can be factorized as the product of a metaplectic operator corresponding to the symplectic matrix defining the phase function, and a Shubin pseudodifferential operator. 
A further consequence are phase space estimates of the FBI transform of the kernels of FIOs. 

The class of FIOs we study is closed under composition and adjoint, and contains the metaplectic group. 
Our composition result Proposition \ref{prop:composition} generalizes the particular case of H\"ormander's composition theorem \cite[Proposition~5.9]{Hormander2} when the phase functions are real. 
We extend this special case in \cite{Hormander2} to non-trivial amplitudes. 
This approach is quite different from the techniques used in \cite{Asada1,Helffer1}. 

In the recent paper \cite{Cappiello3} we apply the results from this paper to initial value Cauchy problems for Schr\"odinger type evolution equations with Hamiltonian given by the sum of a real homogeneous quadratic form and a pseudodifferential perturbation from  the Shubin class.

The paper is organized as follows. 
In Section \ref{sec:prelim} 
we recall an FBI type transform, and some basic facts on Shubin pseudodifferential operators and metaplectic operators. 
Then we study oscillatory integrals with Shubin amplitudes and quadratic real-valued phase functions in Section \ref{sec:oscint}. 
In Section \ref{sec:FIO} we define the FIOs that we study, and we compare our assumptions to \cite{Asada1, Helffer1}. 
We show that FIOs are closed under composition, which leads to 
the central result that every FIO can be factored as the composition of a metaplectic operator and a Weyl pseudodifferential operator of Shubin type. 
Section \ref{sec:phasechar} is devoted to phase space analysis of kernels of FIOs in terms of estimates on the FBI transform. We define $\Gamma$-Lagrangian distributions in the Shubin framework in Section \ref{sec:Lagrangian} and 
we study the microlocal properties of these distributions, the action of FIOs on them, 
and phase space estimates of the FBI transform. 
In Section \ref{sec:kernelLagrangian} we prove that the Schwartz kernels of the FIOs are identical to the $\Gamma$-Lagrangian distributions associated with the twisted graph Lagrangian.

\section{Preliminaries}\label{sec:prelim}

\subsection*{Basic notation}

An open ball in $\rr d$ is denoted $B_r(y) = \{ x \in \rr d: \ |x-y| < r \} \subseteq \rr d$ for $y \in \rr d$ and $r>0$, and we write $B_r(0) = B_r$. 
The gradient operator with respect to $x \in \rr d$ is denoted $\nabla_x$, and we write $\nabla_x f(x) = f_x'(x)$. 
We use $\cS(\rr d)$ and $\cS'(\rr d)$ to denote the Schwartz space of rapidly decaying smooth functions, and its dual the tempered distributions, respectively. 
We write 
$(f,g)$ for the sesquilinear pairing, conjugate linear in the second argument, between a distribution $f$ and a test function $g$, as well as the $L^2$-standard scalar product if $f,g \in L^2(\rr d)$.

The symbols $T_{x_0}u(x)=u(x-x_0)$ and $M_\xi u(x) = e^{i \la x,\xi \ra}u(x)$, where $\la \cdot,\cdot \ra$ denotes the inner product on $\rr d$, are used for translation by $x_0\in \rr d$ and modulation by $\xi \in \rr d$, respectively, applied to functions or distributions. For $x\in\rr{d}$ we use $\eabs{x}=\sqrt{1+|x|^2}$. Peetre's inequality is
\begin{equation}\label{eq:Peetre}
\eabs{x+y}^s \leqs C_s \eabs{x}^s\eabs{y}^{|s|}\qquad x,y\in\rr{d}, \quad C_s>0, \quad s \in \ro. 
\end{equation}
We write $\dbar x = (2\pi)^{-d}\dd x$ for the dual Lebesgue measure.
The notation $f (x) \lesssim g(x)$ means $f(x) \leqs C g(x)$ for some $C>0$ for all $x$ in the domain of $f$ and of $g$. 
If $f (x) \lesssim g (x) \lesssim f (x)$ then we write $f (x) \asymp g (x)$. 

The Fourier transform for $f \in \cS(\rr d)$ is normalized as 
\begin{equation*}
\cF f (\xi) = \widehat f (\xi) = (2\pi)^{-d/2} \int_{\rr d} f(x) e^{-i \la x,\xi \ra} \, \dd x 
\end{equation*}
which makes it unitary on $L^2(\rr d)$. 
The partial Fourier transform with respect to a vector variable indexed by $j$ is denoted $\cF_j$. 

The orthogonal projection on a linear subspace $Y \subseteq \rr d$ is denoted $\pi_Y$. 
We denote by $\M_{d_1 \times d_2}( \ro )$ the space of $d_1\times d_2$ matrices with real entries, by $\GL(d,\ro ) \subseteq \M_{d \times d}( \ro )$ the group of invertible matrices, and by $\On(d) \subseteq \GL(d,\ro)$ the subgroup of orthogonal matrices.
The determinant of $A \in \M_{d \times d}( \ro )$ is $|A|$. 
If $f$ is a function on $\rr d$ and $A \in \GL(d,\ro)$ the pullback is denoted $A^* f = f \circ A$.

\subsection*{An integral transform of FBI type} 

The following integral transform has been used extensively in \cite{Cappiello2} and is fundamental also for this article. 
For more information see \cite{Cappiello2}. 
\begin{defn}
Let $u\in \cS^\prime(\rr d)$ and let $g\in \cS(\rr d)\setminus\{0\}$. The transform $u \mapsto \cT_g u$ is 
\begin{equation*}
\cT_g u(x,\xi)=(2\pi)^{-d/2}(u,T_x M_{\xi}g), \quad x, \xi \in \rr d. 
\end{equation*}
\end{defn}

If $u \in \cS(\rr d)$ then $\cT_g u \in \cS(\rr {2d})$ by \cite[Theorem~11.2.5]{Grochenig1}. 
The adjoint $\cT_g^*$ is defined by $(\cT_g^* U, f) = (U, \cT_g f)$ for $U \in \cS'(\rr {2d})$ and $f \in \cS(\rr d)$. 
When $U$ is a polynomially bounded measurable function we write
\begin{equation*}
\cT_g^* U(y) = (2\pi)^{-d/2} \int_{\rr {2d}} U(x,\xi) \, T_{x} M_{\xi} g(y) \, \dd x \, \dd \xi 
\end{equation*}
where the integral is defined weakly so that $(\cT_g^* U, f) = (U, \cT_g f)_{L^2}$ for $f \in \cS(\rr d)$. 

\begin{prop}
\label{prop:Swdchar}
\rm{\cite[Theorem~11.2.3]{Grochenig1}}
Let $u\in\cS'(\rr d)$ and let $g \in \cS(\rr d) \setminus 0$. Then $\cT_g u\in C^\infty(\rr {2d})$ and there exists $N \in \no$ such that %
\begin{equation*}
|\cT_g u(x,\xi)|\lesssim \eabs{(x,\xi)}^{N}, \quad (x,\xi) \in \rr {2d}.
\end{equation*}
We have $u\in \cS(\rr d)$ if and only if for any $N \geqs 0$ 
\begin{equation*}
|\cT_g u(x,\xi)|\lesssim \eabs{(x,\xi)}^{-N}, \quad (x,\xi) \in \rr {2d}. 
\end{equation*}
\end{prop}

The transform $\cT_g$ is closely related to the short-time Fourier transform \cite{Grochenig1}
\begin{equation*}
\cV_g u(x,\xi) = (2\pi)^{-d/2}(u,M_{\xi} T_x g), \quad x, \xi \in \rr d, 
\end{equation*}
namely $\cT_g u(x,\xi) = e^{i \la x, \xi \ra} \cV_{g}u(x,\xi)$.
If $g,h\in \cS(\rr d)$ then
\begin{equation*}
\cT_h^*\cT_g u=(h,g) u, \qquad u \in \cS'(\rr d), 
\end{equation*}
and thus $\|g\|_{L^2}^{-2}\cT_g^*\cT_g u=u$ for $g \in \cS(\rr d) \setminus 0$, cf. \cite{Grochenig1}.

Finally we recall the definition of the Gabor wave front set, cf. \cite{Hormander1,Rodino1}. 

\begin{defn}\label{def:WFG}
If $u \in \cS'(\rr d)$ and $g \in \cS(\rr d) \setminus 0$ then $z_0 \in T^*\rr d  \setminus 0$ satisfies  $z_0 \notin \WF(u)$ if 
there exists an open cone $V \subseteq T^* \rr d \setminus 0$ containing $z_0$, such that for any $N \in \no$ there exists $C_{V,g,N}>0$ such that $|\cT_g u(z)|\leqs C_{V,g,N} \eabs{z}^{-N}$ when $z \in V$.
\end{defn}

The Gabor wave front set is hence a closed conic subset of $T^*\rr d  \setminus 0$. 

\subsection*{Weyl pseudodifferential operators and metaplectic operators} 

We use pseudodifferential operators in the Weyl calculus. 
Such an operator is defined by a symbol $a$ defined on $\rr {2d}$ as 
\begin{equation*}
a^w(x,D) f(x) = \int_{\rr {2d}} e^{i \la x-y, \xi \ra} a\left((x+y)/2,\xi \right) \, f(y) \, \dbar \xi \, \dd y. 
\end{equation*}
We will later use Shubin symbols, but for now it suffices to note that the Weyl quantization extends 
by the Schwartz kernel theorem to $a \in \cS'(\rr {2d})$, and then gives 
rise to a continuous linear operator $a^w(x,D): \cS(\rr d) \to \cS'(\rr d)$. 
The space of Weyl pseudodifferential operators with symbols in a space $U \subseteq \cS'(\rr {2d})$ is denoted $\OP^w U$. 

The Weyl product $a \wpr b$ is the product on the symbol level corresponding to composition of operators,  
\begin{equation*}
(a \wpr b )^w(x,D) = a^w(x,D) b^w(x,D)
\end{equation*}
when the composition is well defined.
The (Schwartz) kernel of the operator $a^w(x,D)$ is
\begin{equation}\label{eq:schwartzkernelpseudo}
K_{a}(x,y)=\int_{\rr d} e^{i \la x-y, \xi \ra} a\left((x+y)/2,\xi \right) \dbar \xi
\end{equation}
interpreted as a partial inverse Fourier transform of $a$, followed by a change of variables, when $a \in \cS'(\rr {2d})$. 

For $a \in \cS'(\rr {2d})$ and $f,g \in \cS(\rr d)$ we have
\begin{equation}\label{eq:wignerweyl}
(a^w(x,D) f,g) = (2 \pi)^{-d/2} (a, W(g,f) ) 
\end{equation}
where 
\begin{equation}\label{eq:wignerdistribution}
W(g,f) (x,\xi) = (2 \pi)^{-d/2} \int_{\rr d} g(x+y/2) \overline{f(x-y/2)} \, e^{- i \la y, \xi \ra} \, \dd y \in \cS(\rr {2d})
\end{equation}
is the Wigner distribution \cite{Folland1,Grochenig1}. 

We view $T^* \rr d$ as a symplectic vector space equipped with the 
canonical symplectic form
\begin{equation}\label{eq:cansympform}
\sigma((x,\xi), (x',\xi')) = \la x' , \xi \ra - \la x, \xi' \ra, \quad (x,\xi), (x',\xi') \in T^* \rr d.
\end{equation}
A Lagrangian (subspace) \cite{Hormander0} is a linear subspace $\Lambda \subseteq T^* \rr d$ of dimension $d$ such that 
\begin{equation*}
\sigma(X,Y) = 0, \quad X,Y \in \Lambda. 
\end{equation*}

The real symplectic group $\Sp(d,\ro)$ is the set of matrices in $\GL(2d,\ro)$ that leaves $\sigma$ invariant. 
To each $\chi \in \Sp(d,\ro)$ is associated an operator $\mu(\chi)$ which is unitary on $L^2(\rr d)$, and determined up to a complex factor of modulus one, such that
\begin{equation}\label{eq:metaplecticoperator}
\mu(\chi)^{-1} a^w(x,D) \, \mu(\chi) = (a \circ \chi)^w(x,D), \quad a \in \cS'(\rr {2d})
\end{equation}
(cf. \cite{Folland1,Hormander0}).
The operators $\mu(\chi)$ are homeomorphisms on $\mathscr S$ and on $\mathscr S'$, 
and are called metaplectic operators.

The metaplectic representation is the mapping $\Sp(d,\ro) \ni \chi \mapsto \mu(\chi)$ which is a homomorphism modulo sign
\begin{equation}\label{eq:metaplechomo}
\mu(\chi_1) \mu(\chi_2) = \pm \mu(\chi_1 \chi_2), \quad \chi_1,\chi_2 \in \Sp(d,\ro). 
\end{equation}
Two ways to overcome the sign ambiguity are to pass to a double-valued representation \cite{Folland1}, or to a representation of the so called two-fold covering group of $\Sp(d,\ro)$. 
The latter group is called the metaplectic group $\Mp(d,\ro)$. 
The two-to-one projection $\pi: \Mp(d,\ro) \rightarrow \Sp(d,\ro)$ is $\mu(\chi)\mapsto \chi$ whose kernel is $\{\pm 1\}$. The sign ambiguity may be fixed (hence it is possible to choose a section of $\pi$) along a continuous path $\ro \ni t\mapsto \chi_t \in \Sp(d,\ro)$. This involves the so called Maslov factor \cite{Leray1}.

Let $\psi_0 = \pi^{-d/4} e^{-|x|^2/2}$, $x \in \rr d$. 
A localization operator \cite{Nicola1} with symbol $a \in \cS'(\rr {2d})$ is defined by 
\begin{equation*}
(A_a u, f) = (a \cTp, \cTp f), \quad u,f \in \cS(\rr d). 
\end{equation*}
We have (cf. \cite[Section~1.7.2]{Nicola1}) $A_a = b^w(x,D)$ where 
\begin{equation*}
b = \pi^{-d} e^{-|\cdot|^2} * a. 
\end{equation*}
%

\section{Oscillatory integrals with respect to quadratic phase functions and Shubin amplitudes}
\label{sec:oscint}

In this section we study oscillatory integrals of the form 
\begin{equation}\label{eq:kernelFIO}
K_{\varphi,a}(x,y) = \int_{\rr N} e^{i \fy(x,y,\theta)} a(x,y,\theta) \, \dd \theta, \quad x,y \in \rr d. 
\end{equation}
They will later be used as kernels of FIOs. 

We make the following assumptions on the phase function $\fy$. 
It is a real-valued quadratic form on $\rr {2d+N}$, 
\begin{equation}\label{eq:pform}
\fy(x,y,\theta) = \la (x,y, \theta), \Phi (x, y, \theta) \ra, \quad x,y \in \rr d, \quad \theta \in \rr N, 
\end{equation}
where $\Phi \in \M_{(2d+N) \times (2d+N)} (\ro)$ is symmetric. 
We decompose $\Phi$ into blocks as
\begin{equation}\label{eq:Phimatrix}
\Phi= \frac1{2} \left(
\begin{array}{cc}
F & L \\
L^t & Q
\end{array}
\right)
\end{equation}
where $F \in \M_{2d \times 2d}(\ro)$, $L \in \M_{2d \times N}(\ro)$ and $Q \in \M_{N \times N}(\ro)$, and where $F$ and $Q$ are symmetric.
Thus 
\begin{equation*}
\fy(x,y,\theta) = \frac{1}{2} \la (x,y), F (x, y) \ra + \la L \theta, (x,y) \ra + \frac{1}{2} \la \theta, Q \theta \ra, \quad (x,y,\theta) \in \rr {2d+N}.
\end{equation*}

We assume the following non-degeneracy condition: 
\begin{equation}\label{eq:fullrank}
\mbox{The submatrix }
\left(
\begin{array}{c}
L \\
Q
\end{array}
\right) \in \M_{(2d+N) \times N}(\ro)
\mbox{ is injective.}
\end{equation}

As example is given by the pseudodifferential operator phase function $\fy(x,y,\xi) $ $= \la x-y, \xi\ra$
where $F=0$, $Q=0$ and
\begin{equation*}
L
=\left(
\begin{array}{c}
I_d \\
-I_d
\end{array}
\right) \in \M_{2d \times d}(\ro).
\end{equation*}

In \eqref{eq:kernelFIO} we assume $N \geqs 0$. 
If $N=0$ then the matrices $L$ and $Q$ do not exist and we interpret the integral \eqref{eq:kernelFIO} as 
\begin{equation}\label{eq:kernelFIOdegen}
K_{\varphi,a}(x,y) = e^{i \fy(x,y)} a(x,y)  \quad x,y \in \rr d. 
\end{equation}

Denote $X=(x,y) \in \rr {2d}$. 
The \emph{critical set} defined by a phase function $\fy$ is the linear subspace
\begin{equation*}
C_\fy = \{ (X,\theta) \in \rr {2d+N}: \  \fy_\theta' (X,\theta) = 0 \} \subseteq \rr {2d+N} 
\end{equation*}
and the associated Lagrangian subspace is 
\begin{equation*}
\Lambda_\fy = \{ (X,\fy_X' (X,\theta)) \in T^* \rr {2d}: \  \fy_\theta' (X,\theta) = 0 \} 
\subseteq T^* \rr {2d}. 
\end{equation*}
Owing to the properties of $\fy$ we have $\dim C_\fy = \dim \Lambda_\fy = 2d$.

The amplitude $a$ in \eqref{eq:kernelFIO} is assumed to be of Shubin type \cite{Shubin1}.
Let $\Omega \subseteq \rr {2d+N}$ be open and let $0 \leqs \rho \leqs 1$.  
The space of Shubin amplitudes of order $m \in \ro$ is denoted $\Gamma_\rho^m(\Omega)$, and 
$a \in \Gamma_\rho^m(\Omega)$ means that $a \in C^\infty(\Omega)$ and 
\begin{equation}\label{eq:shubinestimate}
|\partial_X^\alpha \partial_\theta^\beta a(X,\theta)| 
\lesssim \eabs{(X,\theta)}^{m-\rho|\alpha+\beta|}, \quad (\alpha, \beta) \in \nn {2d +N}, \quad (X,\theta) \in \Omega. 
\end{equation} 
We denote $\Gamma^m(\Omega)=\Gamma_1^m(\Omega)$ 
and $\Gamma_\rho^\infty(\Omega) = \bigcup\limits_{m \in \ro} \Gamma_\rho^m(\Omega)$. 

We will mostly assume $a \in \Gamma_\rho^m(\rr {2d+N})$. 
Occasionally we will discuss a larger space of amplitudes, introduced by Helffer \cite{Helffer1}, that is adapted to a given phase function. 
Consider for $\ep>0$ the open conic set 
\begin{equation}\label{eq:coneneighborhood}
V_{\fy,\ep} = \{ (X,\theta) \in \rr {2d+N}: \ | \fy_\theta' (X,\theta)| < \ep |(X,\theta)| \} \subseteq \rr {2d+N}
\end{equation}
which contains the critical set $C_\fy$. We denote  
\begin{equation}\label{eq:helfferamplitude}
\Gamma^m_{\fy,\rho,\ep}(\rr {2d+N}) = \Gamma_\rho^m (V_{\fy,\ep}) \cap \Gamma^{\infty}_0(\rr {2d+N}),
\end{equation}
see \cite[Section 2.2]{Helffer1}.
In this definition it is required that the amplitudes behave like $\Gamma_\rho^m$ only in a conic neighborhood of the critical set, outside of which a cruder polynomial estimate of the derivatives is sufficient. 
The space $\Gamma^m_{\fy,\rho,\ep}(\rr {2d+N})$ includes symbols of pseudodifferential operators (cf. Remark \ref{rem:pseudohelffer}). 

\begin{rem}
The restriction to quadratic phase functions is crucial in order to obtain estimates in the phase space for the FBI transform of the kernels.
However, the conditions \eqref{eq:pform}, \eqref{eq:Phimatrix}, \eqref{eq:fullrank} combined with Shubin amplitudes are less restrictive than one might think. In fact we can allow a phase function of the form $\fy_r = \fy + r$ where $\fy$ satisfies the assumptions above and $r \in \Gamma^0_\rho(\rr {2d+N})$. Then $e^{ir} \in \Gamma^0_\rho (\rr {2d+N})$, and hence the factor $e^{ir}$ can be absorbed into the amplitude. 
This means that the phase function only has to be a non-degenerate quadratic form modulo an element in $\Gamma^0_\rho (\rr {2d+N})$. 
\end{rem}

We note that the conditions \eqref{eq:pform}, \eqref{eq:Phimatrix}, \eqref{eq:fullrank} on the phase function are neither weaker nor stronger than the conditions in \cite{Asada1}, and the same observation holds for the conditions in \cite{Helffer1}. 
In  \cite{Asada1,Helffer1} the authors deal with phase functions that are more general than quadratic forms, but their conditions of non-degeneracy are stronger than ours. On the other hand, we shall later restrict to phase functions that are associated to the twisted graph Lagrangian of a symplectic matrix. Under this additional condition our phase functions become a subset of the ones considered in \cite{Asada1, Helffer1}, cf. Remark \ref{rem:AH}. 

If $a \in \Gamma_\rho^m(\rr {2d+N})$ with $m < - N$ the integral \eqref{eq:kernelFIO} converges absolutely 
and defines a polynomially bounded function.
Due to the properties of $\fy$ it is possible to give meaning to \eqref{eq:kernelFIO} for any $m \in \ro$. 
To wit, by the regularization procedure described in \cite[Section~5]{Hormander2} and \cite[Section~3]{PRW1}, one extends \eqref{eq:kernelFIO} to $m \in \ro$ obtaining a kernel $K_{\varphi,a} \in \cS'(\rr {2d})$. 

More precisely, first let $a \in \Gamma_\rho^m(\rr {2d+N})$ with $m < -N$ and let $f \in \cS(\rr {2d})$. 
For $1 \leqs j \leqs N$ we have
\begin{align*}
(K_{\varphi,a},f) & = \int_{\rr {2d + N}} e^{i \fy(X,\theta)} a(X,\theta) \, \overline{f(X)} \, \dd \theta \, \dd X \\
& = \int_{\rr {2d +N}} e^{i \fy(X,\theta)} P_j \Big( a(X,\theta) \, \overline{f(X)}  \Big) \, \dd \theta \, \dd X
\end{align*}
where $P_j$ is a first order differential operator of the form
\begin{equation*}
P_j g= \left( 1 + \langle u_j, \nabla_{X,\theta} \rangle + \langle b_j,X \rangle \right) \left( \frac{g}{1-i \theta_j} \right) , \quad u_j \in \rr {2d+N}, \quad b_j \in \rr {2d},  
\end{equation*}
acting on $g \in C^\infty(\rr {2d+N})$. 
Iterating this $k$ times and then over $1 \leqs j \leqs N$ produces
\begin{equation}\label{eq:regularization}
(K_{\fy,a},f)  = \int_{\rr {2d +N}} e^{i \fy(X,\theta)} P \Big( a(X,\theta) \, \overline{f(X)} \Big) \, \dd \theta \, \dd X
\end{equation}
where 
\begin{equation}\label{eq:Pdef}
P = P_1^k P_2^k \cdots P_N^k. 
\end{equation}

For $m \in \ro$ and $a \in \Gamma_\rho^m(\rr {2d+N})$ the integral \eqref{eq:regularization} converges and defines a distribution in $\cS'(\rr {2d})$ provided $k \geqs |m| + 2$, 
since the factors in the denominator $1-i \theta_j$ make the integral with respect to $\theta$ convergent. 
Thus $K_{\fy, a} \in \cS'(\rr {2d})$ is well defined for $a \in \Gamma_\rho^m(\rr {2d+N})$ where $m \in \ro$ is arbitrary, 
and the extension is unique. 
Equivalently we may define $K_{\varphi,a} \in \cS'(\rr {2d})$ for $a \in \Gamma_\rho^m(\rr {2d+N})$ as
\begin{equation*}
(K_{\fy,a},f)  = \lim_{\ep \rightarrow 0+}\int_{\rr {2d + N}} \chi_\ep(\theta) \, e^{i \fy(X,\theta)} a(X,\theta) \, \overline{f(X)} \, \dd \theta \, \dd X, \quad f \in \cS(\rr {2d}), 
\end{equation*}
where $\chi_\ep(\theta) = \chi(\ep \theta)$, $\chi \in \cS(\rr N)$, $\ep>0$ and $\chi(\theta) = 1$ when $|\theta| \leqs 1$.
The latter regularization can be written as
\begin{equation*}
K_{\fy,a}(X)  = \lim_{\ep \rightarrow 0+}\int_{\rr N} \chi_\ep(\theta) \, e^{i \fy(X,\theta)} a(X,\theta) \, \dd \theta, \quad X \in \rr {2d}.
\end{equation*}

In the following we show that the oscillatory integral \eqref{eq:kernelFIO} with $a \in \Gamma_\rho^m(\rr {2d+N})$ may be rewritten with a possibly new amplitude $b \in \Gamma_\rho^m(\rr {2d+n})$ for some $n \in \no$ such that $0 \leqs n \leqs N$ and a possibly new phase function which lacks the term $Q$, cf. \eqref{eq:pform} and \eqref{eq:Phimatrix}. 

\begin{prop}\label{prop:phasereduction}
Suppose $N \geqs 1$, $a \in \Gamma_\rho^m(\rr {2d+N})$ and let $\fy$ be a quadratic phase function defined by a symmetric matrix $\Phi \in \M_{(2d+N) \times (2d+N)} (\ro)$ denoted as in \eqref{eq:Phimatrix} and satisfying \eqref{eq:fullrank}. 
Denote the corresponding Lagrangian by $\Lambda_\fy \subseteq T^* \rr {2d}$. 

Then there exists $n \in \no$ such that $0 \leqs n \leqs N$, and \eqref{eq:kernelFIO} can be written as the oscillatory integral 
\begin{equation*}
K_{\varphi_0, b}(X) = \int_{\rr n} e^{i \fy_0(X,\theta)} b(X,\theta) \, \dd \theta, \quad X \in \rr {2 d},  
\end{equation*}
where $b \in \Gamma_\rho^m(\rr {2d+n})$, and where the new phase function $\fy_0$ is defined by a symmetric matrix $\Phi_0 \in \M_{(2d+n) \times (2d+n)} (\ro)$ denoted as in \eqref{eq:Phimatrix} with $Q=0$ and satisfying \eqref{eq:fullrank}. 
Furthermore $\fy_0$ parametrizes the same Lagrangian as $\fy$, that is $\Lambda_{\fy_0}=\Lambda_\fy$. 
\end{prop}

\begin{proof}
By an orthogonal change of variables in the integral \eqref{eq:kernelFIO} we may assume that $Q = \diag(q_1, \cdots,q_N)$ is diagonal, modifying the amplitude $a \in \Gamma_\rho^m(\rr {2d+N})$ without altering $\Lambda_\fy$. 

In the following we suppose that $N > 1$, but the argument holds also for $N=1$, with natural modifications (vectors in $\rr {N-1}$, matrices in $\M_{2d \times (N-1)}(\ro)$ and functions on $\rr {N-1}$ 
are interpreted as non-existing). 

Denote $L = [L_0 \ \ell]$ where $L_0 \in \M_{2d \times (N-1)}(\ro)$, $\ell \in \M_{2d \times 1}(\ro)$ and $Q_0 = \diag(q_1, \cdots,q_{N-1})$. 
Suppose $q_N = \partial_{\theta_N}^2 \fy \neq 0$.
Denoting $\theta = (\theta',\theta_N) \in \rr N$ with $\theta' \in \rr {N-1}$ and $\theta_N \in \ro$, 
completion of the square gives 
\begin{equation*}
\fy(X,\theta) = \frac{1}{2} q_N (\theta_N + q_N^{-1} \la \ell,X \ra )^2 + \fy_0(X,\theta')
\end{equation*}
where 
\begin{equation*}
\fy_0(X,\theta') = \frac{1}{2} \la X, F_0 X \ra + \la L_0 \theta', X \ra + \frac{1}{2} \la \theta', Q_0 \theta' \ra
\end{equation*}
and 
\begin{equation*}
F_0 = F - q_N^{-1} \ell \ell^t. 
\end{equation*}
The condition \eqref{eq:fullrank} is preserved for the matrices that define $\fy_0$. 

Denote the Lagrangian corresponding to $\fy_0$ by $\Lambda_{\fy_0}$. 
Suppose $(X,\theta) \in \rr {2d+N}$ and $(X,FX+L \theta) \in \Lambda_\fy$, that is $L^t X + Q \theta = 0$. 
Then $L_0^t X + Q_0 \theta' = 0$ and $\la \ell,X \ra + q_N \theta_N = 0$ which gives $F_0 X + L_0 \theta' = F X + L \theta$. 
Thus $\Lambda_\fy \subseteq \Lambda_{\fy_0}$ and hence $\Lambda_\fy =\Lambda_{\fy_0}$ since both are subspaces of dimension $2d$.

Set $c=q_N/2$ and $u = - q_N^{-1} \ell$. 
Let $\chi \in \cS(\rr {N-1})$ satisfy $\chi (\theta') = 1$ when $|\theta'| \leqs 1$, 
and let $\psi \in \cS(\ro)$ satisfy $\psi(\theta_N) = 1$ when $|\theta_N| \leqs 1$. 
We have for $f \in \cS(\rr {2d})$
\begin{equation}\label{eq:Kfiteratedint}
\begin{aligned}
(K_{\varphi,a},f)   
& = \lim_{\ep \rightarrow 0+} \int_{\rr {2d + N}} (\chi \otimes \psi)_\ep (\theta) \, e^{i \fy(X,\theta)} a(X,\theta) \, \overline{f(X)} \, \dd \theta \, \dd X \\
& = \lim_{\ep \rightarrow 0+} \int_{\rr {2d + N-1}} \chi_\ep (\theta') \, e^{i \fy_0(X,\theta')} b_\ep(X, \theta') \, \overline{f(X)} \, \dd \theta' \, \dd X \\
& = \lim_{\ep \rightarrow 0+} \int_{\rr {2d + N-1}} e^{i \fy_0(X,\theta')} \, P_0 \left( \chi_\ep (\theta') \, b_\ep(X, \theta') \, \overline{f(X)} \right) \dd \theta' \, \dd X
\end{aligned}
\end{equation}
where 
\begin{multline}\label{eq:bepdef}
b_\ep(X, \theta') :=  \int_\ro  \psi_\ep (\theta_N) \, e^{i c (\theta_N - \la u,X \ra )^2}  \, a(X,\theta) \, \dd \theta_N  \\
= \int_\ro  \psi_\ep (\theta_N + \la u,X \ra) \, e^{i c \, \theta_N ^2}  \, a(X,\theta', \theta_N + \la u,X \ra) \, \dd \theta_N
\end{multline}
and where $P_0$ is an operator that corresponds to $\fy_0$ as $P$ corresponds to $\fy$ in \eqref{eq:regularization} and \eqref{eq:Pdef} with $k \in \no$ sufficiently large. 
For fixed $\ep>0$ the function $b_\ep \in C^\infty(\rr {2d+N-1})$ satisfies the estimates
\begin{equation*}
\left| \pdd X \alpha \pdd {\theta'} \beta b_\ep (X,\theta') \right|
\lesssim \eabs{(X,\theta')}^{m-\rho|\beta|}, 
\quad (\alpha,\beta) \in \nn {2d+N-1}, \quad (X,\theta') \in \rr {2d+N-1}, 
\end{equation*}
which are slightly different from the Shubin estimates \eqref{eq:shubinestimate}. 
However, they suffice to make sense of the oscillatory integral \eqref{eq:Kfiteratedint}. 

Our plan is to show first
\begin{equation}\label{eq:blimit}
b(X,\theta') =  \lim_{\ep \rightarrow 0+}  b_{\ep} (X,\theta') \in \Gamma_\rho^m(\rr {2d+N-1}), 
\end{equation}
and then 
\begin{equation}\label{eq:P0blimit}
\lim_{\ep \rightarrow 0+} P_0 ( \chi_\ep (\theta') \, b_{\ep} (X, \theta') \, \overline{f(X)})
=  P_0 ( b (X, \theta') \, \overline{f(X)}), \quad (X,\theta') \in \rr {2d+N-1}, 
\end{equation}
and finally
\begin{equation}\label{eq:bboundunif}
| P_0 ( \chi_\ep (\theta') \, b_{\ep} (X, \theta') \, \overline{f(X)})| \lesssim \eabs{(X,\theta')}^{-2d - N+1}, \quad (X,\theta') \in \rr {2d+N-1}, 
\end{equation}
uniformly over $0 < \ep \leqs 1$. 
The limit \eqref{eq:P0blimit} inserted into \eqref{eq:Kfiteratedint}, combined with the estimate \eqref{eq:bboundunif} and dominated convergence then give 
\begin{align*}
(K_{\varphi,a},f) 
& = \int_{\rr {2d + N-1}} e^{i \fy_0(X,\theta')} \, P_0 \left( b(X, \theta') \, \overline{f(X)} \right) \dd \theta' \, \dd X \\
& = \lim_{\ep \rightarrow 0+} \int_{\rr {2d + N-1}} \chi_\ep (\theta') \, e^{i \fy_0(X,\theta')} \, b(X, \theta') \, \overline{f(X)} \, \dd \theta' \, \dd X.
\end{align*}
By induction this proves the theorem. 

It remains to show  \eqref{eq:blimit}--\eqref{eq:bboundunif}. 
We start with \eqref{eq:blimit}. 
Let $R>0$ and set 
\begin{align*}
b_{1,\ep} (X,\theta') & = \int_\ro \psi(\theta_N/R) \, \psi_\ep (\theta_N + \la u,X \ra) \, e^{i c \, \theta_N^2 } a(X, \theta', \theta_N + \la u,X \ra) \, \dd \theta_N, \\
b_{2, \ep} (X,\theta') & = \int_\ro (1-\psi(\theta_N/R)) \, \psi_\ep (\theta_N + \la u,X \ra) \, e^{i c \, \theta_N^2 } a(X, \theta', \theta_N + \la u,X \ra ) \, \dd \theta_N
\end{align*}
so that $b_\ep = b_{1,\ep} + b_{2,\ep}$. 
It is clear that 
\begin{equation}\label{eq:b1konv}
\begin{aligned}
b_1 (X, \theta') 
& := \lim_{\ep \rightarrow 0+}  b_{1,\ep} (X, \theta') = \int_\ro  \psi(\theta_N/R) e^{i c \, \theta_N ^2}  \, a(X,\theta', \theta_N + \la u,X \ra) \, \dd \theta_N \\
& \quad \in \Gamma_\rho^m(\rr {2d+N-1}) 
\end{aligned}
\end{equation}
and we also have the uniform estimates over $0 < \ep \leqs 1$
\begin{equation}\label{eq:b1bound}
| \pdd X \alpha \pdd {\theta'} \beta b_{1,\ep} (X, \theta') | \lesssim \eabs{(X,\theta')}^{|m|}, \quad (X,\theta') \in \rr {2d+N-1}, 
\quad (\alpha,\beta) \in \nn {2d+N-1}. 
\end{equation}

To estimate $b_{2, \ep}$ we first regularize the integral. 
We have $(-\partial_{\theta_N})^j e^{i c \, \theta_N^2 } = p_j(\theta_N) e^{i c \, \theta_N^2 }$ where $p_j$ is a polynomial of degree $j \in \no$. Integration by parts thus gives 
\begin{align*}
& b_{2,\ep} (X,\theta') = \\
& \int_\ro e^{i c \, \theta_N^2 } 
\partial_{\theta_N}^j \Big( p_j^{-1}(\theta_N) (1-\psi(\theta_N/R)) \, \psi_\ep (\theta_N + \la u,X \ra) \, a(X, \theta', \theta_N + \la u,X \ra ) \Big) \, \dd \theta_N.
\end{align*}
If we first pick $j$ sufficiently large and then $R>0$ sufficiently large (to avoid the zeroes of $p_j$),  
the integral converges and we obtain 
\begin{equation}\label{eq:b2konv}
\begin{aligned}
& b_2 (X, \theta')  := \lim_{\ep \rightarrow 0+}  b_{2,\ep} (X,\theta') \\
& = \int_\ro e^{i c \, \theta_N^2 } 
\partial_{\theta_N}^j \Big( p_j^{-1}(\theta_N) (1-\psi(\theta_N/R)) \, a(X, \theta', \theta_N + \la u,X \ra ) \Big) \, \dd \theta_N \\
& \quad \in \Gamma_\rho^m(\rr {2d+N-1})
\end{aligned}
\end{equation}
as well as the uniform estimates over $0 < \ep \leqs 1$
\begin{equation}\label{eq:b2bound}
| \pdd X \alpha \pdd {\theta'} \beta b_{2,\ep} (X, \theta') | \lesssim \eabs{(X,\theta')}^{|m|}, \quad (X,\theta') \in \rr {2d+N-1}, 
\quad (\alpha,\beta) \in \nn {2d+N-1}.
\end{equation}

Combining \eqref{eq:b1konv} and \eqref{eq:b2konv} proves \eqref{eq:blimit}, 
and we obtain from \eqref{eq:b1bound}, \eqref{eq:b2bound} 
the uniform estimates over $0 < \ep \leqs 1$
\begin{equation}\label{eq:bbound}
| \pdd X \alpha \pdd {\theta'} \beta b_{\ep} (X, \theta') | \lesssim \eabs{(X,\theta')}^{|m|}, \quad (X,\theta') \in \rr {2d+N-1}, 
\quad (\alpha,\beta) \in \nn {2d+N-1}. 
\end{equation}

Next we show \eqref{eq:P0blimit}. Let $(X,\theta') \in \rr {2d+N-1}$ be fixed. 
First we look at the operator $P_j$ defined by its action on $f \in C^\infty(\rr{2d+N-1})$ by
\begin{equation*}
P_j f = \left( 1 + \la u_j, \nabla_X\ra + \langle v_j,X \rangle + \la w_j , \nabla_{\theta'} \ra \right) \left( \frac{f}{1-i \theta_j} \right)
\end{equation*}
for $1 \leqs j \leqs N-1$, 
where $u_j,v_j \in \rr {2d}$ and $w_j \in \rr {N-1}$.
We have
\begin{equation}\label{eq:Pjaction}
\begin{aligned}
& P_j \left( \chi_\ep (\theta') \, b_{\ep} (X, \theta') \, \overline{f(X)} \right) \\
& =  \chi_\ep (\theta')  P_j \left( b_{\ep} (X, \theta') \, \overline{f(X)} \right)
+ \la w_j , \nabla_{\theta'} \ra \left( \chi_\ep (\theta') \right) \frac{b_{\ep} (X, \theta') \, \overline{f(X)}}{1-i\theta_j}. 
\end{aligned}
\end{equation}
Note that $\la w_j , \nabla_{\theta'} \ra \left( \chi_\ep (\theta') \right) = \cO(\ep)$. 
Since $P_0 = P_1^k P_2^k \cdots P_{N-1}^k$ for some $k \in \no$, 
it suffices to show, taking into account \eqref{eq:bbound}, 
\begin{equation*}
\lim_{\ep \rightarrow 0+} P_0 ( b_{\ep} (X, \theta') \, \overline{f(X)})
=  P_0 ( b (X, \theta') \, \overline{f(X)}), \quad (X,\theta') \in \rr {2d+N-1}. 
\end{equation*}
The validity of this identity can be verified by means of the decomposition $b_\ep = b_{1,\ep} + b_{2,\ep}$ above. 
The details are left to the reader. 
Thus \eqref{eq:P0blimit} has been proved. 

Finally we indicate how to show \eqref{eq:bboundunif}.
Again we use the decomposition $b_\ep = b_{1,\ep} + b_{2,\ep}$. 
Combining this with \eqref{eq:Pjaction} and $P_0 = P_1^k P_2^k \cdots P_{N-1}^k$ for $k \in \no$ sufficiently large, 
one can confirm the estimate \eqref{eq:bboundunif}. 
The details are again left to the reader. 
\end{proof}

As a consequence of the proposition we may assume
\begin{equation}\label{eq:phaseredux}
\Phi = \frac{1}{2} \left(
\begin{array}{cc}
F &  L \\
L^t & 0
\end{array}
\right)
\end{equation}
where $F \in \M_{2d \times 2d}(\ro)$ is symmetric and $L \in \M_{2d \times N}(\ro)$ is injective.
The corresponding Lagrangian $\Lambda \subseteq T^* \rr {2d}$ is 
\begin{equation}\label{eq:Lambdaphi}
\Lambda = \{(X, FX + L \theta): \, (X,\theta) \in \rr {2d + N}, \ L^t X = 0\}. 
\end{equation}

Conversely it can be shown (cf. \cite{PRW1}) that any Lagrangian $\Lambda \subseteq T^* \rr {2d}$ can be parametrized in this way, 
for a symmetric matrix $F \in \M_{2d \times 2d}(\ro)$ and an injective matrix $L \in \M_{2d \times N}(\ro)$. 
The matrix $L$ is uniquely determined modulo invertible right factors.  
The matrix $F$ can be assumed to satisfy $\ran F \perp \ran L$ \cite{PRW1}, 
but $F$ is not uniquely determined by $\Lambda$. 
What is unique is $F_Y = \pi_Y F \pi_Y$ where $Y = \Ker L^t$, but $F - F_Y$ can be arbitrary. 

If $\fy_1$ and $\fy_2$ both parametrize a given Lagrangian $\Lambda \subseteq T^* \rr {2d}$ as in \eqref{eq:Lambdaphi}, 
and $a \in \Gamma_\rho^m(\rr {2d+N})$, then 
$K_{\fy_1,a} = K_{\fy_2,a}$ is not guaranteed. 
In fact, if $\fy_j$ is defined by matrices 
$F_j \in \M_{2d \times 2d}(\ro)$, $L_j \in \M_{2d \times N}(\ro)$ and $Q_j \in \M_{N \times N}(\ro)$ for $j=1,2$, then 
by Proposition \ref{prop:phasereduction} 
\begin{equation}\label{eq:kernelreducedphase}
K_{\fy_j,a} (X) = \int_{\rr n} e^{\frac{i}{2} \la X, F_j X \ra + i \la L \theta, X\ra } a_j(X,\theta) \, \dd \theta
\end{equation}
where $L \in \M_{2d \times n}(\ro)$ is injective and $n \leqs N$, 
since after reduction to $Q_j=0$, $j=1,2$, we have $\Ker L_1^t = \Ker L_2^t$. 
Here $a_1 \in \Gamma_\rho^m(\rr {2d+n})$ is not guaranteed to equal $a_2 \in \Gamma_\rho^m(\rr {2d+n})$, and likewise $F_1 \neq F_2$ in general, whereas $\pi_Y F_1 \pi_Y = \pi_Y F_2 \pi_Y$. 

We are interested in phase functions that correspond to twisted graph Lagrangians in $T^* \rr {2d}$ with respect to a symplectic matrix $\chi \in \Sp(d,\ro)$. 
The graph in $T^* \rr d \times T^* \rr d$ with respect to $\chi \in \Sp(d,\ro)$ is
\begin{equation}\label{eq:graphlagrangian}
\Lambda_\chi = \{(x,\xi;y,\eta) \in T^* \rr d \times T^* \rr d: \  (x,\xi) = \chi (y,\eta) \} \subseteq T^* \rr d \times T^* \rr d, 
\end{equation}
and it is a Lagrangian if we equip $T^* \rr d \times T^* \rr d$ with the symplectic form
\begin{equation*}
\sigma_1 (x,y,\xi,\eta)=\sigma(x,\xi) - \sigma(y,\eta), \quad (x,y), \ (\xi,\eta) \in T^* \rr d \times T^* \rr d.
\end{equation*}
The symplectic vector space $(T^* \rr d \times T^* \rr d,\sigma_1)$ is isomorphic to $T^* \rr {2d}$ equipped with the canonical symplectic form \eqref{eq:cansympform}. The isomorphism is given by the twist operator
\begin{equation*}
(x,\xi,y,\eta)' = (x,\xi,y,-\eta), \quad x,y,\xi,\eta \in \rr d, 
\end{equation*}
followed by transposition of the second and third variables.
The twisted graph Lagrangian with respect to $\chi \in \Sp(d,\ro)$ is 
\begin{equation}\label{eq:twistedgraphlagrangian}
\Lambda_\chi' = \{(x, y, \xi,-\eta) \in T^* \rr {2d}: \  (x,\xi) = \chi (y,\eta) \} \subseteq T^* \rr {2d}. 
\end{equation}
\begin{rem}
Note that the notations $\Lambda_\chi \subseteq T^* \rr d \times T^* \rr d$ in \eqref{eq:graphlagrangian} and $\Lambda_\chi'  \subseteq T^* \rr {2d}$ in \eqref{eq:twistedgraphlagrangian} understand different ambient symplectic spaces.
\end{rem}

\begin{defn}\label{def:FIOkernel}
If $m \in \ro$ and $\chi \in \Sp(d,\ro)$ we denote by $K_{\rho}^m(\chi) \subseteq \cS'(\rr {2d})$ the set of kernels $K_{\varphi,a}$ defined as in \eqref{eq:kernelFIO} where $a \in \Gamma_\rho^m(\rr {2d+N})$ for $N \geqs 0$, and where the phase function $\fy$  parametrizes the twisted graph Lagrangian $\Lambda_\chi' \subseteq T^* \rr {2d}$. We write $K_1^m(\chi) = K^m(\chi)$.
\end{defn}

For pseudodifferential operators the Lagrangian is the conormal bundle of the diagonal, that is
$\Lambda = N(\Delta) = \Delta \times \Delta^\perp \subseteq T^* \rr {2d}$, with  
\begin{equation}\label{eq:diagonal}
\Delta  = \{(x,x): \, x \in \rr d \} \subseteq \rr {2d}, \qquad \Delta^\perp  = \{(\xi,-\xi): \, \xi \in \rr d \} \subseteq \rr {2d}.
\end{equation}
This means $N(\Delta)= \Lambda_{I}'$ where $I \in \Sp(d,\ro)$ is the identity matrix.

\section{Fourier integral operators with quadratic phase functions}
\label{sec:FIO}

In this section we treat FIOs defined by kernels that are oscillatory integrals as in Definition \ref{def:FIOkernel} and compare our conditions with \cite{Asada1} and \cite{Helffer1}.

\begin{defn}\label{def:FIO}
Let $\chi \in \Sp(d,\ro)$, $N \geqs 0$ and $a \in \Gamma^m_\rho(\rr {2d+N})$.  
The operator with kernel $K_{\fy,a} \in K_\rho^m(\chi)$ is denoted $\cK_{\fy,a}$ and called FIO. 
The set of operators with kernels in $K_{\rho}^m(\chi)$ is denoted $\cI^m_\rho(\chi)$, and $\cI^m_1(\chi) = \cI^m(\chi)$.
\end{defn}

Thus
\begin{equation*}
(\cK_{\fy,a} f,g) = ( K_{\fy,a}, g \otimes \overline f), \quad f,g \in \cS(\rr d). 
\end{equation*}
Since $K_{\fy,a} \in \cS'(\rr {2d})$ the FIO $\cK_{\fy,a} : \cS(\rr d) \to \cS'(\rr d)$ is continuous. 

The following result appears implicitly in \cite{Helffer1}. We prefer to include it in order to give a self-contained account. 

\begin{lem}\label{lem:regularizing}
Let $N \geqs 0$ and let $\fy$ be a quadratic form defined by  a symmetric $\Phi \in \M_{(2d+N) \times (2d+N)} (\ro)$ as in \eqref{eq:Phimatrix} such that \eqref{eq:fullrank} is satisfied.  
Define $V_{\fy,\ep}$ by \eqref{eq:coneneighborhood}.
If $a \in \Gamma_\rho^m(\rr {2d+N})$ satisfies 
\begin{equation*}
\overline{\supp(a) \cap V_{\fy,\ep}} \ \mbox{\rm is compact in} \ \rr {2d+N}
\end{equation*}
for some $\ep>0$ then $K_{\fy,a} \in \cS(\rr {2d})$.
\end{lem}

\begin{proof}
The case $N=0$ is trivial so we may assume $N \geqs 1$. 
If $\supp(a)$ is compact then $K_{\fy,a} \in \cS(\rr {2d})$. 
Using an appropriate cutoff function we may therefore assume 
$\supp(a) \cap (V_{\fy,\ep} \cup B_r)= \emptyset$ for some $r>0$. 
By Proposition \ref{prop:phasereduction} we may assume that $\fy_\theta' (x,y,\theta)$ is a linear function that does not depend on $\theta \in \rr N$. 

On the support of $a$ we may write
\begin{equation*}
e^{i \fy(x,y,\theta)} = -i |\fy_\theta'|^{-2} \la \fy_\theta' , \nabla_\theta e^{i \fy(x,y,\theta)} \ra.  
\end{equation*}
First we introduce the operator
\begin{equation*}
T_\theta g (\theta) = i \la \nabla_\theta, \fy_\theta'  |\fy_\theta' |^{-2} g(\theta) \ra
=  i |\fy_\theta'|^{-2} \la \fy_\theta',  \nabla g \ra
\end{equation*}
that acts on $g \in C^\infty(\rr N)$ provided $\fy_\theta' \neq 0$,
and integrate by parts. This yields for $f \in \cS(\rr {2d})$ and $n \in \no$
\begin{equation*}
(K_{\fy,a},f) = \lim_{\delta \rightarrow 0+}\int_{\rr {2d + N}} \chi_\delta(\theta) \, e^{i \fy(x,y,\theta)} 
T_\theta^n \, a(x,y,\theta) \, \overline{f(x,y)} \, \dd \theta \, \dd x \, \dd y
\end{equation*}
where $\chi \in \cS(\rr N)$ and $\chi(\theta) = 1$ when $|\theta| \leqs 1$.

The assumption implies that we have in the support of $a$ 
\begin{equation}\label{criticalset1}
\ | \fy_\theta' (x,y,\theta)| \geqs \ep |(x,y,\theta)|.
\end{equation}
From the observation that $\fy_\theta'$ is a linear function it follows that for $n \in \no$ sufficiently large, 
we have
\begin{equation*}
K_{\fy,a}(x,y) = \int_{\rr N} e^{i \fy(x,y,\theta)} 
T_\theta^n \, a(x,y,\theta) \, \dd \theta.
\end{equation*}
The integral is absolutely convergent thanks to $\eqref{criticalset1}$ if $n \in \no$ is sufficiently large. 
The same facts imply that the integral belongs to $\cS(\rr {2d})$. 
\end{proof}

We say that a continuous linear operator $\cK: \cS'(\rr d) \to \cS'(\rr d)$ is regularizing if it is continuous
\begin{equation*}
\cK: \cS'(\rr d) \to \cS(\rr d)
\end{equation*}
which is equivalent to the property of its kernel $K \in \cS(\rr {2d})$.
Hence from Lemma \ref{lem:regularizing} we obtain the following result.

\begin{cor}\label{cor:regularizing}
Under the assumptions of Lemma \ref{lem:regularizing}
the operator $\cK_{\fy,a}$ is regularizing.
\end{cor}

The next result shows that any regularizing operator can be considered an FIO in $\cap_{m \in \ro} \cI^m(\chi)$ for $\chi \in \Sp(d,\ro)$ arbitrary. 

\begin{lem}\label{lem:reginclu}
Let $\chi \in \Sp(d,\ro)$, $N \geqs 0$, and suppose $\fy$ is a quadratic form, defined by a symmetric $\Phi \in \M_{(2d+N) \times (2d+N)} (\ro)$ as in \eqref{eq:Phimatrix} such that \eqref{eq:fullrank} is satisfied, 
that parametrizes the twisted Lagrangian $\Lambda_\chi' \subseteq T^* \rr {2d}$. 
If $K \in \cS(\rr {2d})$ then there exists $a \in \cS(\rr {2d+N})$ such that 
$K = K_{\fy,a}$. 
\end{lem}

\begin{proof}
Again we may assume $N \geqs 1$. 
Let $g \in \cS(\rr N)$ satisfy $\int g(\theta)\,\dd\theta=1$. Then
\begin{equation*}
a(X,\theta) = K(X) g(\theta) e^{-i \fy(X,\theta)} \in\cS(\rr {2d+N})
\end{equation*}
and 
\begin{equation*}
K_{\fy,a}(X) = \int_{\rr N} e^{i \fy(X,\theta)} a(X,\theta) \,\dd\theta =\int_{\rr N} K(X) g(\theta) \,\dd\theta = K(X). 
\end{equation*}
\end{proof}

\begin{rem}\label{rem:pseudohelffer}
The space of amplitudes $\Gamma_\rho^m(\rr {2d+N})$ may seem somewhat restrictive (cf. \cite{Helffer1,Shubin1}). 
For instance the symbol $a(x,\theta) \in \Gamma^m_\rho (\rr {2d})$ of a Kohn--Nirenberg pseudodifferential operator
$a(x,D)$ is not an amplitude of three variables in $\Gamma_\rho^m (\rr {3d})$ since the derivatives of $a(x,\theta)$ do not decay with respect to $y$.
However, by picking a conical cutoff function $\psi \in C^\infty(\rr {3d})$ that is one around the cone $V_{\fy,\ep}$ (except on a compact set) defined by \eqref{eq:coneneighborhood}
for $\ep>0$ sufficiently small, it follows from the proof of \cite[Proposition~2.2.4]{Helffer1} that $a \psi \in \Gamma_\rho^m(\rr {3d})$. The operator with amplitude $a (1-\psi)$ is regularizing by Corollary \ref{cor:regularizing}, 
and by Lemma \ref{lem:reginclu} its amplitude can be absorbed into the amplitude $a \psi \in \Gamma_\rho^m(\rr {3d})$.  
Thus $a(x,D) \in \cI_\rho^m(I)$. 

Helffer's larger space of amplitudes $\Gamma_{\fy,\rho,\ep}^{m}(\rr {2d+N})$ contains symbols from $\Gamma^m_\rho(\rr {2d})$ without modification. 
On the other hand, by \cite[Lemma 2.2.2]{Helffer1} every amplitude $a \in \Gamma_{\fy,\rho,\ep}^{m}(\rr {2d+N})$ can be decomposed as $a = a_1 + a_2$ where $a_1 \in \Gamma_\rho^m(\rr {2d+N})$ and $a_2$ gives rise to a regularizing operator, and also the calculus developed in \cite{Helffer1} is constructed modulo regularizing operators.
Hence, either choice of amplitudes yields the same calculus modulo regularizing operators. We prefer to work with the phase-independent choice of $\Gamma_{\rho}^{m}(\rr {2d+N})$, eliminating the necessity of an additional cut-off argument in certain proofs.
\end{rem}

In the following we study properties of the FIOs 
$\cK_{\fy,a}$ with kernel $K_{\fy,a} \in K_\rho^m(\chi)$ for $\chi \in \Sp(d,\ro)$. 
In view of Remark \ref{rem:pseudohelffer} we may replace the assumption $a \in \Gamma_\rho^m(\rr {2d+N})$ with $a \in \Gamma_{\fy,\rho,\ep}^{m}(\rr {2d+N})$ in all results which hold modulo regularizers.

First we observe that for trivial amplitude a FIO is a metaplectic operator times a nonzero constant, see \cite[Section~5]{Hormander2} for the proof.

\begin{prop}\label{prop:metaplectic}
If $\chi \in \Sp(d,\ro)$ and $\cK_{\fy,1} \in \cI^0(\chi)$ then 
\begin{equation*}
\cK_{\fy,1} = C_\fy \mu(\chi)
\end{equation*}
where $C_\fy \in \co \setminus 0$ depends on the phase function which parametrizes the Lagrangian $\Lambda_\chi'$.
\end{prop}

Next we study the case of non-trivial amplitudes and a particular feature of the matrix $\chi \in \Sp(d,\ro)$. 

We recall that a symplectic matrix has the block form
\begin{equation}\label{eq:symplecticABCD}
\chi = 
\left(
  \begin{array}{cc}
  A & B \\
  C & D
  \end{array}
\right) \in \Sp(d,\ro)
\end{equation}
where $A, B, C, D \in \M_{d \times d}(\ro)$ satisfy
\begin{align}
& A^t C = C^t A, \quad B^t D = D^t B, \quad A B^t = B A^t, \quad C D^t = D C^t, \label{eq:sympmatrix1} \\
& A^t D - C^t B = I, \quad A D^t - B C^t = I, \label{eq:sympmatrix2} 
\end{align}
cf. \cite[Proposition 4.1]{Folland1}. 

A matrix $\chi \in \Sp(d,\ro)$ is called \emph{free}  \cite{deGosson2}
when $B \in \GL(d,\ro)$.
In this case the matrix 
\begin{equation*}
\left(
  \begin{array}{cc}
  A & B \\
  I & 0
  \end{array}
\right) \in \M_{2d \times 2d}(\ro)
\end{equation*}
is invertible with inverse 
\begin{equation*}
\left(
  \begin{array}{cc}
  0 & I \\
  B^{-1} & -B^{-1} A
  \end{array}
\right) \in \GL(2d,\ro).
\end{equation*}
This gives
\begin{align}\notag
\Lambda_\chi' 
& = \left\{ 
(Ay+B \eta, y, C y + D \eta, -\eta ) \in T^*\rr {2d}: \ (y,\eta) \in \rr {2d}
\right\} \\
& = \left\{( X,F X) \in T^*\rr {2d}: \ X \in \rr {2d} \right\}
\label{eq:Ffreeparam}
\end{align}
with 
\begin{equation}\label{eq:Ffree}
F = 
\left(
  \begin{array}{cc}
  C & D \\
  0 & -I
  \end{array}
\right)
\left(
  \begin{array}{cc}
  0 & I \\
  B^{-1} & -B^{-1} A
  \end{array}
\right) 
= \left(
  \begin{array}{cc}
  D B^{-1} & - B^{-t} \\
  -B^{-1} & B^{-1} A
  \end{array}
\right) = F^t
\end{equation}
thanks to the identities \eqref{eq:sympmatrix1} and \eqref{eq:sympmatrix2}. 

The upshot of this is as follows. 
The matrix $\chi \in \Sp(d,\ro)$ is free exactly when the corresponding twisted graph Lagrangian has the form $\Lambda_\chi' = \left\{( X,F X) \in T^*\rr {2d}: \ X \in \rr {2d} \right\}$.
By \eqref{eq:Ffreeparam} we may choose the canonical phase function $\fy(X) = \frac{1}{2} \la X,FX\ra$ to parametrize $\Lambda_\chi'$, which is reduced in the sense that $N=0$. 
The kernel \eqref{eq:kernelFIO} is then interpreted as \eqref{eq:kernelFIOdegen}, that is
\begin{equation}\label{eq:kernelfree}
K_{\fy,a} (X) = e^{\frac{i}{2} \la X, F X \ra}  a(X), \quad X \in \rr {2d}, 
\end{equation}
where $a \in \Gamma_\rho^m(\rr {2d})$. 
Any kernel in $K_{\rho}^m(\chi)$ can be reduced by means of Proposition \ref{prop:phasereduction} to one with kernel of the form \eqref{eq:kernelfree}. 
We note that the matrix $F \in \M_{2d \times 2d}(\ro)$ is uniquely defined by \eqref{eq:Ffree} in terms of the blocks of $\chi$.

Next we prove a result which allows us to compare our conditions on the phase function with the ones assumed in \cite{Asada1,Helffer1}.

\begin{lem}\label{lem:Helffercondition}
Let $\chi \in \Sp(d,\ro)$ and suppose the twisted graph Lagrangian \eqref{eq:twistedgraphlagrangian} is parametrized by 
the quadratic form defined by \eqref{eq:pform} and \eqref{eq:Phimatrix} such that \eqref{eq:fullrank} holds. 
Denote
\begin{equation}\label{eq:FLblocks}
F = \left(
\begin{array}{cc}
E &  G \\
G^t & H
\end{array}
\right), 
\quad 
L = \left(
\begin{array}{c}
P \\
R 
\end{array}
\right), 
\end{equation}
with $E, G, H \in \M_{d \times d} (\ro)$, $E$, $H$ symmetric, and $P,R \in \M_{d \times N} (\ro)$.  
Then 
\begin{equation} \label{invmatr}
\left(
\begin{array}{ccc}
G^t & H & R \\
0 & I & 0 \\
P^t & R^t & Q 
\end{array}
\right), 
\left(
\begin{array}{ccc}
I & 0 & 0 \\
E & G & P \\
P^t & R^t & Q 
\end{array}
\right)
\in \GL(2d+N,\ro)
\end{equation}
where the matrices are interpreted as having cancelled third block row and third block column if $N=0$.
\end{lem}

\begin{proof}
First we assume $N \geqs 1$. 
By permutation of rows and columns in \eqref{invmatr}, and expansion with respect to the identity matrix it can be seen that the matrices have equal determinant. 
It suffices therefore to show that the left matrix in \eqref{invmatr} is invertible. 
Suppose 
\begin{equation*}
(x,y,\theta) \in \Ker \left(
\begin{array}{ccc}
G^t & H & R \\
0 & I & 0 \\
P^t & R^t & Q 
\end{array}
\right), 
\end{equation*}
that is $y=0$ and 
\begin{equation*}
\left\{
\begin{array}{l}
G^t x + R \theta  = 0 \\
P^t x + Q \theta = 0
\end{array}
\right. .
\end{equation*}
Since $\fy_\theta' (x,0,\theta) = P^t x + Q \theta = 0$ we have
\begin{equation*}
(x, \fy_x'(x,0,\theta), 0, - \fy_y' (x,0,\theta) ) \in \Lambda_\chi,  
\end{equation*}
that is $(x, \fy_x' (x,0,\theta) ) = \chi (0, - \fy_y' (x,0,\theta) )$. 

With the notation \eqref{eq:symplecticABCD}
we may thus write with the stipulated matrix notation 
\begin{equation*}
\left\{
\begin{array}{l}
x = - B ( G^t x + R \theta ) = 0 \\
E x  + P \theta =  -D ( G^t x + R \theta ) = 0
\end{array}
\right. .
\end{equation*}
Thus $x=0$, $P \theta=0$, $R \theta=0$ and $Q \theta=0$. 
By \eqref{eq:fullrank} $\theta=0$, which proves  
\begin{equation*}
\left(
\begin{array}{ccc}
G^t & H & R \\
0 & I & 0 \\
P^t & R^t & Q 
\end{array}
\right)
\in \GL(2d+N,\ro). 
\end{equation*}

Finally we discuss the case when $N = 0$ which means that $\chi \in \Sp(d,\ro)$ is free. 
We have to show
\begin{equation*}
\left(
\begin{array}{cc}
G^t & H \\
0 & I  \\
\end{array}
\right), 
\left(
\begin{array}{cc}
I & 0 \\
E & G \\
\end{array}
\right)
\in \GL(2d,\ro).
\end{equation*}
This follows from $G^t, G \in \GL(d,\ro)$ which is a consequence of \eqref{eq:Ffree}.
\end{proof}

\begin{rem}\label{rem:AH}
Lemma \ref{lem:Helffercondition} implies that the non-degeneracy conditions on the phase functions assumed in \cite{Helffer1} are satisfied.
Indeed if $\chi \in \Sp(d,\ro)$ is not free then any phase function $\fy$ parametrizing $\Lambda_\chi'$ satisfies
\begin{equation}\label{eq:Helffercondition}
\begin{aligned}
|(x,y,\theta)| & \lesssim |(\fy_y', y, \fy_\theta' )|, \quad (x,y,\theta) \in \rr {2d+N} \\
|(x,y,\theta)| & \lesssim |(x, \fy_x', \fy_\theta' )|, \quad (x,y,\theta) \in \rr {2d+N}. 
\end{aligned}
\end{equation}
If $\chi \in \Sp(d,\ro)$ is free then $N=0$ may be assumed, and 
\begin{equation}\label{eq:Helfferconditionfree}
\begin{aligned}
|(x,y)| & \lesssim |(\fy_y', y)|, \quad (x,y) \in \rr {2d} \\
|(x,y)| & \lesssim |(x, \fy_x' )|, \quad (x,y) \in \rr {2d}. 
\end{aligned}
\end{equation}
Also the stronger condition
\begin{equation*}
\left|
\begin{pmatrix} \Phi_{x,y}^{''} & \Phi_{x,\theta}^{''} \\ \Phi_{\theta,y}^{''} & \Phi_{\theta, \theta}^{''} \end{pmatrix} \right| \geqs \delta_0 >0
\end{equation*}
assumed in \cite{Asada1} reduces in our case to the invertibility of the matrix 
\begin{equation*}
\begin{pmatrix} G &P \\ R^t & Q \end{pmatrix}
\end{equation*}
which is granted by Lemma \ref{lem:Helffercondition}.
\end{rem}

From Remark \ref{rem:AH} and \cite[Proposition 2.1.1]{Helffer1} we obtain the following consequence. 

\begin{cor}\label{cor:FIOScont}
If $\chi \in \Sp(d,\ro)$ and $\cK \in \cI_{\rho}^m(\chi)$ then 
$\cK$ is continuous on $\cS(\rr d)$ and extends uniquely to be continuous on $\cS' (\rr d)$. 
\end{cor}

Therefore FIOs may be composed as continuous operators on $\cS(\rr d)$. It turns out that the composition is again an FIO associated to the composition of the involved symplectic matrices. 
The following result generalizes H\"ormander's composition theorem \cite[Proposition~5.9]{Hormander2}, restricted to real-valued phase functions, to the case of non-trivial amplitudes.

\begin{prop}\label{prop:composition}
Let $\chi_j \in \Sp(d,\ro)$ and suppose $\cK_j \in \cI_{\rho}^{m_j}(\chi_j)$, for $j=1,2$. 
Then $\cK_1 \cK_2  \in \cI_{\rho}^{m_1+m_2}(\chi_1 \chi_2)$. 
\end{prop}

\begin{proof}
The proof is inspired by that of \cite[Proposition~5.9]{Hormander2} and that of \cite[Proposition 2.2.3]{Helffer1}. 

Corollary \ref{cor:FIOScont} implies that the composition $\cK_1 \cK_2: \cS(\rr d) \to \cS(\rr d)$ is a well defined continuous operator.
First we assume that $N_1, N_2 \geqs 1$, that is, none of $\chi_1,\chi_2 \in \Sp(d,\ro)$ is a free symplectic matrix. 

Let $\fy_j$ be a phase function that parametrizes the twisted graph Lagrangian $\Lambda_{\chi_j}'$ for $j=1,2$, respectively. 
By Proposition \ref{prop:phasereduction} we may assume that $\fy_j$ is a quadratic form defined as in \eqref{eq:pform} and \eqref{eq:Phimatrix} with $F_j \in \M_{2d \times 2d}(\ro)$, $L_j \in \M_{2d \times N_j}(\ro)$ and $Q_j= 0$, for $j=1,2$. 
Each Lagrangian $\Lambda_{\chi_j}'$ has the form \eqref{eq:Lambdaphi}. 
The kernel of $\cK_1 \cK_2$ is
\begin{equation}\label{eq:oscintcomp}
K(x,y) = \int_{\rr {d + N_1+N_2}} e^{i \left( \fy_1(x,z,\theta) + \fy_2(z,y,\xi) \right)} a_1(x,z,\theta) \, a_2(z,y,\xi) \, \dd z \, \dd \theta \, \dd \xi, 
\end{equation}
$x,y \in \rr d$, where we view $(z,\theta,\xi) \in \rr {d+N_1+N_2}$ as the covariable.
First we show that \eqref{eq:oscintcomp} is well defined as an oscillatory integral. 

The amplitude is 
\begin{equation*}
b(x,y,z,\theta,\xi) = a_1(x,z,\theta) \, a_2(z,y,\xi) 
\end{equation*}
which we at first consider an element in $\Gamma_0^{|m_1| + |m_2|} (\rr {3d+N_1+N_2})$. 
The phase function is 
\begin{equation*}
\fy(x,y,z,\theta,\xi) = \fy_1(x,z,\theta) + \fy_2(z,y,\xi)
\end{equation*}
so the corresponding Lagrangian is 
\begin{align*}
\Lambda & = \{ (x,y, \fy_{1,x}' (x,z,\theta), \fy_{2,y}' (z,y,\xi)) \in T^*\rr {2d} : \\
& \qquad  \fy_{1,y}' (x,z,\theta) + \fy_{2,x}' (z,y,\xi) = \fy_{1,\theta}' (x,z,\theta) = \fy_{2,\xi}' (z,y,\xi) = 0 \}.
\end{align*}

Twisting the Lagrangian and suppressing variables give
\begin{equation*}
\Lambda' = \{ (x,\fy_{1,x}', y, -\fy_{2,y}' ) \in T^*\rr d \times T^*\rr d : \  \fy_{1,y}'+ \fy_{2,x}' = \fy_{1,\theta}' = \fy_{2,\xi}' = 0 \}.
\end{equation*}
Since $(z,  \fy_{2,x}', y, -\fy_{2,y}') \in \Lambda_{\chi_2}$
we have
\begin{equation*}
\chi_2(y, -\fy_{2,y}' ) = (z,  \fy_{2,x}') = (z,  -\fy_{1,y}')
\end{equation*}
which gives 
\begin{equation*}
\chi_1 \chi_2(y, -\fy_{2,y}' ) = (x,  \fy_{1,x}')  
\end{equation*}
since $(x,  \fy_{1,x}', z, -\fy_{1,y}') \in \Lambda_{\chi_1}$.
This means that $\Lambda' = \Lambda_{\chi_1 \chi_2}$, and hence $\fy$ parametrizes the twisted graph Lagrangian $\Lambda = \Lambda_{\chi_1 \chi_2}'$.  

Next we verify condition \eqref{eq:fullrank} for the matrix that defines $\fy$, denoted as in \eqref{eq:pform} and \eqref{eq:Phimatrix}
with
$F \in \M_{2d \times 2d}(\ro)$, $L \in \M_{2d \times (d+N_1+N_2)}(\ro)$ and $Q \in \M_{(d+N_1+N_2) \times (d+N_1+N_2)}(\ro)$.
We adopt the block matrix notation \eqref{eq:FLblocks} for $F_j$ and $L_j$, with index $j=1,2$. 

We have 
\begin{equation}\label{eq:blockcomposition}
\left(
\begin{array}{c}
L \\
Q 
\end{array}
\right) 
= 
\left(
\begin{array}{ccc}
G_1 & P_1 & 0 \\
G^t_2 & 0 & R_2 \\
H_1 + E_2  & R_1 & P_2 \\
R_1^t & 0 & 0 \\
P_2^t & 0 & 0 
\end{array}
\right) \in \M_{(3d+N_1+N_2) \times (d+N_1+N_2)}(\ro). 
\end{equation}

By Lemma \ref{lem:Helffercondition}  
\begin{equation*}
\left(
\begin{array}{cc}
G_1 & P_1 \\
R_1^t & 0 
\end{array}
\right) 
\in \GL(d+N_1, \ro), 
\quad
\left(
\begin{array}{cc}
G_2^t& R_2 \\
P_2^t & 0 
\end{array}
\right) 
\in \GL(d+N_2, \ro).
\end{equation*}
This implies the injectivity of the matrix \eqref{eq:blockcomposition}, and it follows that  
\eqref{eq:oscintcomp} is a well defined oscillatory integral, provided $N_1, N_2 \geqs 1$. 

If $N_1+N_2 = 1$ then one of $\chi_1$ or $\chi_2$ is a free symplectic matrix, that is $B \in \GL(d,\ro)$ in the block decomposition \eqref{eq:symplecticABCD}. 
The argument above goes through verbatim, except that some block matrices of the matrix \eqref{eq:blockcomposition} are cancelled when 
one of the matrices $L_j$ is non-existent. 
If $\chi_2$ is free then 
\begin{equation*}
\left(
\begin{array}{c}
L \\
Q 
\end{array}
\right) 
= 
\left(
\begin{array}{cc}
G_1 & P_1 \\
G_2^t & 0 \\
H_1 + E_2 & R_1 \\
R_1^t & 0 
\end{array}
\right)\in \M_{(3d+N_1) \times (d+N_1)}(\ro) 
\end{equation*}
which is injective by the arguments above, and 
similarly the corresponding matrix is injective if $\chi_1$ is free. 
Thus \eqref{eq:oscintcomp} is a well defined oscillatory integral if $N_1+N_2 = 1$. 

Finally we assume $N_1 = N_2 = 0$, that is both $\chi_1$ and $\chi_2$ are free symplectic matrices.
In this case the matrix \eqref{eq:blockcomposition} shrinks to 
\begin{equation*}
\left(
\begin{array}{c}
L \\
Q 
\end{array}
\right) 
= 
\left(
\begin{array}{c}
G_1 \\
G_2^t \\
H_1 + E_2
\end{array}
\right)\in \M_{3d \times d}(\ro) 
\end{equation*}
which is injective since $G_1 = - B_1^{-t} \in \GL(d,\ro)$
due to \eqref{eq:Ffree}, where we use the notation 
\eqref{eq:symplecticABCD} for $\chi_1$ with corresponding block matrices $A_1,B_1,C_1,D_1 \in \M_{d \times d}(\ro)$.
Thus \eqref{eq:oscintcomp} is a well defined oscillatory integral also if $N_1+N_2 = 0$. 

It remains to prove $\cK_1 \cK_2 = \cK_{\varphi,a} + \cR$ where $\cR$ is regularizing,  and
$a \in \Gamma^{m_1+m_2}_{\rho}(\rr {3d + N_1 + N_2})$.
In fact, by Lemma \ref{lem:reginclu} this would entail 
$\cK_{\varphi,a} + \cR = \cK_{\varphi,c} \in \cI_{\rho}^{m_1+m_2}(\chi_1 \chi_2)$ for a modified amplitude $c \in \Gamma^{m_1+m_2}_{\rho}(\chi_1 \chi_2)$. 

Let $\ep>0$ be arbitrary. 
Calculating modulo a regularizing $\cR$ we use Lemma \ref{lem:regularizing}
to multiply the amplitude $b \in \Gamma_0^{|m_1| + |m_2|} (\rr {3d+N_1+N_2})$ with a smooth conical cut-off function $\psi \in C^\infty(\rr {3d+N_1+N_2})$ such that
\begin{equation*}
a = b \psi \in \Gamma_0^{|m_1| + |m_2|} (\rr {3d+N_1+N_2})
\end{equation*}
has support contained in 
\begin{equation*}
V_{\fy,\ep} = \{ (x,y,z,\theta,\xi) \in \rr {3d+N_1+N_2}: \ | \fy_{z,\theta,\xi}' (x,y,z,\theta,\xi)| < \ep |(x,y,z,\theta,\xi)| \}.
\end{equation*}

Following the proof of \cite[Proposition 2.2.3]{Helffer1}, it can be seen that this implies $a \in \Gamma_\rho^{m_1 + m_2} (\rr {3d+N_1+N_2})$ when $\ep>0$ is sufficiently small. 
This is shown by showing 
\begin{equation*}
\eabs{(x,z,\theta)} \asymp \eabs{(z,y,\xi)} \asymp \eabs{(x,y,z,\theta,\xi)} 
\end{equation*}
when $(x,y,z,\theta,\xi) \in \supp (a)$. 
\end{proof}

The next result concerns the formal adjoint of an FIO.

\begin{prop}\label{prop:formaladjoint}
If $\chi \in \Sp(d,\ro)$, $N \geqs 0$, $a \in \Gamma_\rho^m(\rr {2d + N})$ and $\cK_{\fy,a} \in \cI_\rho^m(\chi)$
then $\cK_{\fy,a}^* = \cK_{\psi,b} \in \cI_\rho^m(\chi^{-1})$
where $\psi(x,y,\theta)=-\varphi(y,x,\theta)$ and $b(x,y,\theta) = \overline{a(y,x,\theta})$. 
\end{prop}

\begin{proof}
By definition the formal adjoint satisfies
\begin{equation*}
(\cK_{\fy,a} f, g) = (f, \cK_{\fy,a}^* g), \quad f,g \in \cS(\rr d). 
\end{equation*}
The left hand side is the oscillatory integral 
\begin{equation*}
\int_{\rr {2d+N}} e^{i \fy(x,y,\theta)} a(x,y,\theta) \, f(y) \, \overline{g(x)} \, \dd \theta \, \dd x \, \dd y 
\end{equation*}
from which it follows that $\cK_{\fy,a}^*$ has kernel
\begin{equation*}
K(x,y) =  \int_{\rr {N}} e^{i \psi(x,y,\theta)} b(x,y,\theta)  \, \dd \theta. 
\end{equation*}
It remains to show that the phase function $\psi(x,y,\theta) = -\fy(y,x,\theta)$ parametrizes $\Lambda_{\chi^{-1}}'$. 
The Lagrangian corresponding to the phase function $\psi$ is
\begin{align*}
\Lambda_\psi
& = \{ (x,y, \psi_x'(x,y,\theta), \psi_y'(x,y,\theta) ) \in T^* \rr {2d} : \ \psi_\theta'(x,y,\theta) = 0 \} \\
& = \{ (x,y, -\fy_y'(y,x,\theta), -\fy_x'(y,x,\theta) ) \in T^* \rr {2d} : \ \fy_\theta'(y,x,\theta) = 0 \} \\
& = \{ (x,y,\xi,\eta) \in T^* \rr {2d} : \ (y,x,-\eta, -\xi) \in \Lambda_\chi' \} \\
& = \{ (x,y,\xi,\eta) \in T^* \rr {2d} : \ (x,\xi) = \chi^{-1} (y,-\eta) \} \\
& = \Lambda_{\chi^{-1}}'. 
\end{align*}
\end{proof}

\subsection{Factorization of FIOs}

In this section we prove that the FIOs admit a factorization into metaplectic and pseudodifferential operators, see \cite[Theorem~1.3]{CGNR} for a related result where modulation spaces are used for amplitudes. 

We need a preparatory result that will be useful also later. 
Here $z=(z_1,z_2) \in \rr {2d}$ where $z_1,z_2 \in \rr d$. 

\begin{lem}\label{lem:FBIsymplectic}
If $u \in \cS'(\rr d)$, $\chi \in \Sp(d,\ro)$ and $g \in \cS(\rr d) \setminus 0$ then for all $(x,\xi) \in T^* \rr d$
\begin{equation*}
\cT_{\mu(\chi) g} (\mu(\chi) u) (x,\xi)
= e^{\frac{i}{2} \left(  \la x, \xi \ra - \la \chi^{-1}(x,\xi)_1, \chi^{-1}(x,\xi)_2 \ra \right)} \cT_g u(\chi^{-1} (x,\xi)). 
\end{equation*}
\end{lem}

\begin{proof}
If for fixed $x,\xi \in \rr d$ we define
\begin{equation*}
a_{x,\xi} (y,\eta) = e^{-\frac{i}{2} \la x,\xi \ra + i (\la \xi,y \ra - \la x,\eta \ra)}, \quad y,\eta \in \rr d, 
\end{equation*}
then $a_{x,\xi}^w(x,D) = T_x M_\xi$. 
Decomposing $\chi \in \Sp(d,\ro)$ into blocks as in \eqref{eq:symplecticABCD}
the inverse is 
\begin{equation}\label{eq:sympinvABCD}
\chi^{-1} = 
\left(
  \begin{array}{cc}
  D^t & -B^t \\
  -C^t & A^t
  \end{array}
\right) \in \Sp(d,\ro), 
\end{equation}
cf. \cite{Folland1}.
This gives
\begin{align*}
a_{x,\xi} ( \chi(y,\eta)) 
& = e^{- \frac{i}{2} \la x,\xi \ra + i (\la \xi,A y + B \eta \ra - \la x, C y + D \eta \ra)} \\
& = e^{- \frac{i}{2} \la x,\xi \ra + i (\la y, A^t \xi - C^t x \ra - \la \eta, D^t x - B^t \xi \ra)} \\
& = e^{-\frac{i}{2} \left( \la x,\xi \ra - \la \chi^{-1}(x,\xi)_1, \chi^{-1}(x,\xi)_2 \ra \right)}
a_{\chi^{-1}(x,\xi)} (y,\eta). 
\end{align*}

Using the symplectic invariance \eqref{eq:metaplecticoperator} we obtain finally
\begin{align*}
\cT_{\mu(\chi) g} (\mu(\chi) u) (x,\xi)
& = (2 \pi)^{-d/2} (u, \mu(\chi)^{-1} T_x M_\xi \mu(\chi) g ) \\
& = (2 \pi)^{-d/2} (u, (a_{x,\xi} \circ \chi)^w(x,D) g ) \\
& = e^{\frac{i}{2} \left( \la x,\xi \ra - \la \chi^{-1}(x,\xi)_1, \chi^{-1}(x,\xi)_2 \ra \right)}
(2 \pi)^{-d/2} (u, a_{\chi^{-1}(x,\xi)}^w(x,D) g ) \\
& = e^{\frac{i}{2} \left(  \la x, \xi \ra - \la \chi^{-1}(x,\xi)_1, \chi^{-1}(x,\xi)_2 \ra \right)} \cT_g u(\chi^{-1} (x,\xi)). 
\end{align*}
\end{proof}

Concerning factorization of FIOs we first treat the case $\chi=\J$, where 
\begin{equation} \label{symplJ}
\J = 
\left(
\begin{array}{cc}
0 & I_d \\
-I_d & 0 
\end{array}
\right) \in \Sp(d,\ro)
\end{equation}
which appears frequently in symplectic linear algebra \cite{Folland1}. 
Notice that $\J$ is free and 
\begin{equation} \label{eq:metaplecticFourier}
\mu(\J ) = \cF.
\end{equation}

\begin{lem}\label{lem:Jfactorization}
If $\cK \in \cI_{\rho}^m(\J)$
then there exists $b \in \Gamma_\rho^m(\rr {2d})$ such that
\begin{equation*}
\cK = b^w(x,D)  \mu(\J) = \mu(\J) (b \circ \J)^w(x,D). 
\end{equation*}
\end{lem}

\begin{proof}
Let $K \in K^m_\rho(\J)$.  
Since $\J$ is free, we may assume after a reduction of fibre variables, that for some $a \in \Gamma_\rho^m(\rr {2d})$
\begin{align*}
K(x,y) & = e^{-i \la x,y \ra}  a(x,y), \quad x,y \in \rr d, 
\end{align*}
using \eqref{eq:Ffree} and \eqref{eq:kernelfree}. 
Likewise the phase function $\fy(x,y) = \la x,y \ra$, $x,y \in \rr d$, parametrizes the Lagrangian $\Lambda_{-\J}'$. 
Thus $K_1(x,y)=e^{i \la x,y \ra} \in K^0(-\J)$. 
Denoting by $\cK$ and $\cK_1$ the operators with kernels $K$ and $K_1$ respectively, 
the kernel of $\cK \cK_1$ is therefore
\begin{equation*}
K_0(x,y) = \int_{\rr d} e^{i \la x-y,z \ra}  a(x,-z) \, dz. 
\end{equation*}
After a change from left to Weyl quantization (cf. \cite[Theorem~23.1]{Shubin1}), this is the kernel of a Weyl operator $b^w(x,D)$ with $b \in \Gamma_\rho^m(\rr {2d})$.
Since $\cK_1^{-1} = (2\pi)^{-d/2}\cF$ by Proposition \ref{prop:metaplectic} we obtain
\begin{equation*}
\cK = (2\pi)^{-d/2}  b^w(x,D)  \cF= (2\pi)^{-d/2}  b^w(x,D)  \mu(\J)
\end{equation*}
which is the first claimed factorization. 
The second claimed factorization follows from the first and \eqref{eq:metaplecticoperator}. 
\end{proof}

\begin{rem}\label{rem:quantization}
Note that \cite[Theorem~23.1]{Shubin1} shows that the symbol space $\Gamma_\rho^m(\rr {2d})$ is independent of quantization (Weyl, Kohn--Nirenberg, or a parametrized set comprising the two) provided $\rho>0$. 
That is, if $a(x,D) = b^w(x,D)$ then $a \in \Gamma_\rho^m(\rr {2d})$ if and only if $b \in \Gamma_\rho^m(\rr {2d})$. 
In the proof of Lemma \ref{lem:Jfactorization} we use the same result also for $\rho=0$. 
It can be motivated as follows. (The argument also gives a short alternative proof of the invariance for $0 < \rho \leqs 1$.)
Suppose $a(x,D) = b^w(x,D)$ and $a \in \Gamma_\rho^m(\rr {2d})$. 
We must show $b \in \Gamma_\rho^m(\rr {2d})$. 
We have (cf. \cite[Theorem~18.5.10]{Hormander0})
$b = \cF^{-1} M \cF a$ where $M$ is the multiplication operator $f (x,\xi) \mapsto e^{- i \la x, \xi \ra/2} f(x,\xi)$. 
Thus $M = \mu(\chi_1)$ where 
\begin{equation*}
\chi_1 = \left(
\begin{array}{ll}
I_{2d} & 0 \\
F & I_{2d}
\end{array}
\right) \in \Sp(2d,\ro)
\end{equation*}
and 
\begin{equation*}
F = - \frac{1}{2} \left(
\begin{array}{ll}
0 & I_d \\
I_d & 0
\end{array}
\right) \in \M_{2d \times 2d}(\ro). 
\end{equation*}
From \eqref{eq:metaplecticFourier} it follows that $b = \pm \mu(\chi_2)a$ with 
\begin{equation*}
\chi_2 = -\J \chi_1 \J = \left(
\begin{array}{ll}
I_{2d} & -F \\
0 & I_{2d}
\end{array}
\right) \in \Sp(2d,\ro). 
\end{equation*}
Lemma \ref{lem:FBIsymplectic} gives with $g \in \cS(\rr {2d}) \setminus 0$ 
\begin{equation*}
\cT_{\mu(\chi_2) g} b (z,\zeta)
= \pm e^{-\frac{i}{2} \la F \zeta, \zeta \ra} \cT_g a(z + F \zeta, \zeta), \quad z,\zeta \in \rr {2d}. 
\end{equation*}
The claim $b \in \Gamma_\rho^m(\rr {2d})$ is now a consequence of \cite[Proposition~2.2]{Cappiello2}. 
\end{rem}

As a consequence of Proposition \ref{prop:composition} and Lemma \ref{lem:Jfactorization} we get a representation theorem for an FIO as the composition of a Weyl pseudodifferential operator and a metaplectic operator. 

\begin{thm}\label{thm:repFIO}
If $\chi \in \Sp(d,\ro)$ and $\cK \in \cI^m_{\rho}(\chi)$ then 
there exist $b \in \Gamma_\rho^m(\rr {2d})$ such that
\begin{equation*}
\cK  = b^w(x,D)  \mu(\chi) = \mu(\chi) (b \circ \chi)^w(x,D). 
\end{equation*}
Conversely, for any $b \in \Gamma_\rho^m(\rr {2d})$ we have $b^w(x,D) \mu(\chi)\in \cI^m_{\rho}(\chi)$.
\end{thm}

\begin{proof}
Let $\cK_{\fy,1} \in \cI^0((-\J \chi)^{-1})$. 
By Proposition \ref{prop:composition} and Lemma \ref{lem:Jfactorization} we have, since $\chi = \J (-\J \chi)$, 
\begin{equation*} 
\cK \cK_{\fy,1} 
= \cK_2
= b^w(x,D)  \mu(\J)
\end{equation*}
where $\cK_2 \in \cI_\rho^m(\J)$ and $b \in \Gamma_\rho^m(\rr {2d})$. 

Proposition \ref{prop:metaplectic} gives $\cK_{\fy,1} ^{-1} = C \mu (-\J \chi)$ where $C \in \co\setminus 0$, and hence
\begin{equation*} 
\cK
= C b^w(x,D)  \mu(\J) \mu(-\J \chi) 
= \pm C b^w(x,D) \mu(\chi).
\end{equation*}
This proves the first claimed factorization. 
The second claimed factorization is again an immediate consequence of \eqref{eq:metaplecticoperator}. 

For the converse implication we observe that Proposition \ref{prop:metaplectic} implies $\mu(\chi) = \cK_{\fy,a} \in \cI^0(\chi)$ for an appropriate phase function $\fy$ and a constant amplitude $a \equiv C \in \co \setminus 0$. 
We also have $b^w(x,D) \in \cI^m_\rho(I)$ (cf. Remark \ref{rem:pseudohelffer}). 
Proposition \ref{prop:composition} then gives $b^w(x,D)  \mu(\chi) \in \cI^m_{\rho}(\chi)$.
\end{proof}

The factorization in Theorem \ref{thm:repFIO} has several consequences.

It means that we could define the FIOs as the operators of the form $b^w(x,D)\mu(\chi)$, cf. \cite{CGNR}. 
Composing two FIOs gives using \eqref{eq:metaplecticoperator}
\begin{align*}
b_1^w(x,D)\mu(\chi_1)b_2^w(x,D)\mu(\chi_2)
& = b_1^w(x,D)(b_2\circ \chi_1^{-1})^w(x,D)\mu(\chi_1)\mu(\chi_2) \\
& = \pm(b_1 \wpr (b_2 \circ \chi_1^{-1}))^w(x,D) \mu(\chi_1\chi_2).
\end{align*}
The FIOs can hence be identified with the semidirect product of Weyl quantized pseudodifferential operators with the metaplectic group. 

Since metaplectic operators and pseudodifferential operators are both continuous on $\cS(\rr d)$, Theorem \ref{thm:repFIO} also gives an alternative proof of continuity of operators in $\cI^m_\rho (\chi)$ on $\cS(\rr d)$ and on $\cS'(\rr d)$. 
We can also deduce the continuity from the Shubin--Sobolev space
$Q^s(\rr d)$ to $Q^{s-m}(\rr d)$ for $s \in \ro$. These spaces were introduced by Shubin \cite{Shubin1} (cf. \cite{Grochenig1,Nicola1}). 
The space $Q^s(\rr d)$ is identical to the modulation space $M^{2}_s(\rr d)$, that is 
\begin{equation*}
Q^s (\rr d)  = \{u \in \cS'(\rr d): \, \eabs{ \cdot }^s \cT_g u \in L^2(\rr {2d}) \} 
\end{equation*}
where $g \in \cS(\rr d) \setminus 0$ is fixed and arbitrary, with norm
\begin{equation*}
\| u \|_{Q^s} = \left\| \eabs{ \cdot }^s \cT_g u \right\|_{L^2(\rr {2d})}. 
\end{equation*}

Since metaplectic operators are homeomorphisms on $Q^s(\rr d)$, cf. \cite[Proposition 400]{deGosson2}, and since 
pseudodifferential operators of order $m$ are continuous from $Q^{s}(\rr d)$ to $Q^{s-m}(\rr d)$ \cite[Theorem~25.2]{Shubin1}, we get the following result.

\begin{prop}\label{prop:contSobolevShubin}
Suppose $\chi \in \Sp(d,\ro)$, and $\cK \in \cI_{\rho}^m(\chi)$. 
Then $\cK: Q^{s}(\rr d) \rightarrow Q^{s-m}(\rr d)$ is continuous for all $s \in \ro$.
\end{prop}

Finally Theorem \ref{thm:repFIO} and \eqref{eq:metaplecticoperator} imply the following result of Egorov type.

\begin{cor}
If $\chi \in \Sp(d,\ro)$, $\cK \in \cI_{\rho}^m(\chi)$ and $b \in \Gamma_\rho^n(\rr {2d})$ then 
there exist $c \in \Gamma_\rho^m(\rr {2d})$ such that
\begin{equation*}
\cK^* b^w(x,D) \cK  = (\overline{c} \wpr (b \circ \chi) \wpr c)^w(x,D) \in \OP^w \Gamma_\rho^{2m+n}. 
\end{equation*}
\end{cor}

\section{Phase space characterization of FIOs}\label{sec:phasechar}

In this section we characterize the kernels of FIOs with estimates on their FBI transform, generalizing  our results for Shubin pseudodifferential operators \cite{Cappiello2}. 

First we show that Theorem \ref{thm:repFIO} gives the following result as by-product. 

\begin{prop}\label{prop:FIOSTFT}
If $\chi \in \Sp(d,\ro)$, $K_{\fy,a} \in K_\rho^m(\chi)$ and $g \in \cS(\rr {d}) \setminus 0$ then 
there exist $b \in \Gamma_\rho^m(\rr {2d})$ and $h \in \cS(\rr {2d}) \setminus 0$ such that 
\begin{align*} 
\cT_{g \otimes g} K_{\fy,a} (z,\zeta)  
& =  e^{\frac{i}{2} \left( \la z_2, \zeta_2\ra + \la \chi(z_2,-\zeta_2)_1, \chi(z_2,-\zeta_2)_2 \ra\right) } \\
& \quad \times \cT_h K_{b} (z_1, \chi(z_2,-\zeta_2)_1, \zeta_1, - \chi(z_2,-\zeta_2)_2), \\
& \qquad \qquad \qquad (z,\zeta) \in T^*\rr {2d}, 
\end{align*}
where $K_b$ is the kernel of $b^w(x,D)$ (cf. \eqref{eq:schwartzkernelpseudo}). 
\end{prop}

\begin{proof}
By Theorem \ref{thm:repFIO} we have for some $b \in \Gamma_\rho^m(\rr {2d})$ 
\begin{equation}\label{eq:FIOkernelSTFT1}
\begin{aligned} 
\cT_{g \otimes g} K_{\fy,a} (z,\zeta) 
& = (2\pi)^{-d} ( K_{\fy,a}, T_z M_\zeta (g \otimes g) ) \\
& = (2\pi)^{-d} ( K_{\fy,a}, T_{z_1} M_{\zeta_1} g \otimes T_{z_2} M_{\zeta_2} g ) \\
& = (2\pi)^{-d} ( \cK_{\fy,a} T_{z_2} M_{-\zeta_2} \overline{g}, T_{z_1} M_{\zeta_1} g ) \\
& = (2\pi)^{-d} ( b^w(x,D)  \mu(\chi) T_{z_2} M_{-\zeta_2} \overline{g}, T_{z_1} M_{\zeta_1} g ) \\
& = (2\pi)^{-d} ( K_{b}, T_{z_1} M_{\zeta_1} g \otimes \overline{\mu(\chi) T_{z_2} M_{-\zeta_2} \overline{g}}). 
\end{aligned}
\end{equation}

Denoting $g_\chi = \mu(\chi) \overline{g} \in \cS(\rr d) \setminus 0$ we study
\begin{equation*} 
\mu(\chi) T_{z_2} M_{-\zeta_2} \overline{g}
= \mu(\chi) T_{z_2} M_{-\zeta_2} \mu(\chi^{-1}) g_\chi. 
\end{equation*}

We have by the proof of Lemma \ref{lem:FBIsymplectic}
\begin{align*} 
\mu(\chi) T_{z_2} M_{-\zeta_2} \mu(\chi^{-1})
=  e^{\frac{i}{2} \left( \la z_2, \zeta_2\ra + \la \chi(z_2,-\zeta_2)_1, \chi(z_2,-\zeta_2)_2 \ra\right) }
T_{\chi(z_2,-\zeta_2)_1} M_{\chi(z_2,-\zeta_2)_2}.
\end{align*}
Inserted into \eqref{eq:FIOkernelSTFT1} this gives finally
\begin{align*} 
& \cT_{g \otimes g} K_{\fy,a} (z,\zeta) \\
& = (2\pi)^{-d} e^{\frac{i}{2} \left( \la z_2, \zeta_2\ra + \la \chi(z_2,-\zeta_2)_1, \chi(z_2,-\zeta_2)_2 \ra\right) } 
( K_{b}, T_{z_1} M_{\zeta_1} g \otimes \overline{ T_{\chi(z_2,-\zeta_2)_1} M_{\chi(z_2,-\zeta_2)_2}  g_\chi}) \\
& = (2\pi)^{-d} e^{\frac{i}{2} \left( \la z_2, \zeta_2\ra + \la \chi(z_2,-\zeta_2)_1, \chi(z_2,-\zeta_2)_2 \ra\right) } 
( K_{b}, T_{(z_1,\chi(z_2,-\zeta_2)_1)} M_{(\zeta_1,-\chi(z_2,-\zeta_2)_2)} (g \otimes \overline{g_\chi} ) ) \\
& = e^{\frac{i}{2} \left( \la z_2, \zeta_2\ra + \la \chi(z_2,-\zeta_2)_1, \chi(z_2,-\zeta_2)_2 \ra\right) } 
\cT_h K_{b} (z_1,\chi(z_2,-\zeta_2)_1, \zeta_1,-\chi(z_2,-\zeta_2)_2)
\end{align*}
with $h = g \otimes \overline{g_\chi}$. 
\end{proof}

As a consequence we obtain
\begin{equation}\label{eq:FBIkernelFIO}
\begin{aligned} 
\cT_{g \otimes g} K_{\fy,a} (z,\zeta)  
& =  e^{\frac{i}{2} \left( \la z, \zeta \ra + \sigma(\chi(z_2,-\zeta_2), (z_1,\zeta_1 ) ) \right) } \\
&\times \cT_h^\Delta K_{b} (z_1, \chi(z_2,-\zeta_2)_1, \zeta_1, - \chi(z_2,-\zeta_2)_2) 
\end{aligned}
\end{equation}
where we use the notation of \cite[Definition~3.2]{Cappiello2}
for $K \in\cS'(\rr {2d})$ and $h \in \cS(\rr {2d}) \setminus 0$
\begin{equation*}
\cT_h^\Delta K (z, \zeta) = e^{-\frac{i}{2} \la \zeta_1-\zeta_2, z_1-z_2 \ra }\cT_h K (z,\zeta), \quad ( z,\zeta) \in T^* \rr {2d},
\end{equation*}
and the symplectic form \eqref{eq:cansympform}. 

Defining 
\begin{equation*} 
\cT_{g \otimes g}^\chi K_{\fy,a} (z,\zeta)  = e^{-\frac{i}{2} \left( \la z, \zeta \ra + \sigma(\chi(z_2,-\zeta_2), (z_1,\zeta_1 ) ) \right) } \cT_{g \otimes g}K_{\fy,a} (z,\zeta)
\end{equation*}
we have thus 
\begin{equation*} 
\cT_{g \otimes g}^\chi K_{\fy,a} (z,\zeta)  = \cT_{h}^\Delta K_{b} (z_1, \chi(z_2,-\zeta_2)_1, \zeta_1, - \chi(z_2,-\zeta_2)_2) 
\end{equation*}
where $h = g \otimes \overline{g_\chi}$. 
When $\chi = I$ we recover $\cT_{g \otimes g}^I K_{\fy,a} (z,\zeta) = \cT_{g \otimes g}^\Delta K_{b} (z,\zeta)$. 

Using the block matrix notation \eqref{eq:symplecticABCD}
we obtain for $(y,\eta) \in T^* \rr d$ 
\begin{align*} 
& \la (\chi(y,\eta), y,-\eta), \nabla_{(z_1,\zeta_1,z_2,\zeta_2)} \ra \cT_h ^\Delta K_{b}  \left(  (z_1, \chi(z_2,-\zeta_2)_1, \zeta_1, - \chi(z_2,-\zeta_2)_2) \right) \\
& = \la \chi(y,\eta), \left( (\nabla_1+\nabla_2, \nabla_3-\nabla_4) \cT_h ^\Delta K_{b} \right) (z_1, \chi(z_2,-\zeta_2)_1, \zeta_1, - \chi(z_2,-\zeta_2)_2) \ra
\end{align*}
where $\nabla_j$ denotes the gradient with respect to the $\rr d$ variable indexed by $j=1,2,3,4$. 

Combined with \cite[Proposition~3.3]{Cappiello2} this gives the following characterization of the kernels of FIOs (cf. \cite{Tataru}). 
Note that we recover \cite[Proposition~3.3]{Cappiello2} when $\chi=I$. 

\begin{thm}\label{thm:FIOkernelchar}
Let $K \in \cS'(\rr {2d})$ and $g \in \cS(\rr {d}) \setminus 0$. Then $K \in K^m_\rho(\chi)$ with $\chi \in \Sp(d,\ro)$ if and only if 
the estimates 
\begin{equation*}
\begin{aligned}
| L_1 \cdots L_k 
\cT_{g \otimes g}^\chi K (z, \zeta)| 
& \lesssim \eabs{(z_1,\zeta_1) + \chi(z_2,-\zeta_2)}^{m-\rho k} \eabs{(z_1,\zeta_1) - \chi(z_2,-\zeta_2))}^{-N}, \\
& \qquad  \qquad ( z,\zeta) \in T^* \rr {2d}, 
\end{aligned}
\end{equation*}
hold for all $k,N \in \no$, where 
\begin{equation*}
L_j = \langle A_j, \nabla_{z,\zeta} \rangle 
\end{equation*}
and $A_j \in \Lambda_\chi'$ for $j=1,2,\dots,k$. 
\end{thm}

Since $\dist^2((z,\zeta), \Lambda_\chi') \asymp |(z_1,\zeta_1) - \chi(z_2,-\zeta_2))|^2$ where $\dist$ denotes Euclidean distance 
between a point and a subspace, 
and $\Lambda_{-\chi}' \subseteq T^* \rr {2d}$ is transversal to $\Lambda_\chi' \subseteq T^* \rr {2d}$ (cf. \cite[p.~11]{Cappiello2})
we can formulate the estimates as 
\begin{equation*}
\begin{aligned}
| L_1 \cdots L_k \cT_{g \otimes g}^\chi K (z, \zeta)| 
& \lesssim ( 1 + \dist((z,\zeta), \Lambda_{-\chi}') )^{m-\rho k} \, ( 1+\dist((z,\zeta), \Lambda_\chi') )^{-N}, \\
& \qquad ( z, \zeta) \in T^* \rr {2d}, 
\end{aligned}
\end{equation*}
where $k,N \in \no$. 

Theorem \ref{thm:FIOkernelchar} implies the following result which generalizes \cite[Corollary~4.18]{Cappiello2}.

\begin{prop}\label{prop:WFkernel}
If $\chi \in \Sp(d,\ro)$ and $K_{\fy,a} \in K^m_\rho(\chi)$ then 
\begin{equation*}
\WF( K_{\fy,a} ) \subseteq \Lambda_\chi' \setminus 0 \subseteq T^* \rr {2d} \setminus 0. 
\end{equation*}
\end{prop}

\begin{proof}
Let $0 \neq (z_0,\zeta_0) \notin \Lambda_\chi'$. 
For some $C>0$ we then have $(z_0,\zeta_0) \in V$ where the open conic set $V \subseteq T^* \rr {2d} \setminus 0$ is defined by
\begin{equation*}
V = \{ (z,\zeta) \in T^* \rr {2d} \setminus 0: \, |(z_1,\zeta_1) + \chi(z_2,-\zeta_2) | < C |(z_1,\zeta_1) - \chi(z_2,-\zeta_2)| \}. 
\end{equation*}
The conclusion is now a consequence of Theorem \ref{thm:FIOkernelchar} with $k=0$, and 
\begin{equation*}
|(z,\zeta)|^2 \asymp | (z_1,\zeta_1) + \chi(z_2,-\zeta_2) |^2 +  | (z_1,\zeta_1) - \chi(z_2,-\zeta_2) |^2. 
\end{equation*}
\end{proof}

Combining Proposition \ref{prop:WFkernel} with \cite[Proposition~2.11]{Hormander1} we obtain the following result on propagation of Gabor singularities. 
An alternative proof can be given by combining Theorem \ref{thm:repFIO} with \cite[Proposition~2.9 and Eq.~(2.18)]{Rodino1}. 

\begin{cor}\label{cor:FIOmicrolocal}
If $\chi \in \Sp(d,\ro)$ and $\cK \in \cI^m_\rho(\chi)$ then 
\begin{equation*}
\WF( \cK u ) \subseteq \chi \WF(u), \quad u \in \cS'(\rr d). 
\end{equation*}
\end{cor}

\begin{rem}
More precisely the statement holds for the Sobolev--Gabor wave front set for any $s \in \ro$ (cf. \cite{SW2}), as
\begin{equation*}
\WF_{Q^{s-m}}( \cK u ) \subseteq \chi \WF_{Q^s}(u), \quad u \in \cS'(\rr d). 
\end{equation*}
\end{rem}

\section{$\Gamma$-Lagrangian distributions}\label{sec:Lagrangian}

Here we introduce Lagrangian distributions adapted to the Shubin calculus. 
For simplicity we work in Sections \ref{sec:Lagrangian} and \ref{sec:kernelLagrangian} with $\rho=1$ but all results are true with natural modifications if $0 \leqs \rho \leqs 1$. 
Before giving a precise definition we need some preliminary steps. 

Let $\Lambda \subseteq T^* \rr d$ be a Lagrangian.
Referring to Section \ref{sec:oscint} we can write
\begin{equation}\label{eq:Lagrangian}
\Lambda = \{ (X, FX + Z) \in T^* \rr d, \ X \in Y, \ Z \in Y^\perp \} 
\end{equation}
where $Y \subseteq \rr d$ is a linear subspace and $F \in \M_{d \times d}( \ro )$ is a symmetric matrix that leaves $Y$ invariant \cite{PRW1}. 
It then automatically leaves $Y^\perp$ invariant so can be written
\begin{equation*}
F = F_Y + F_{Y^\perp}
\end{equation*}
where $F_Y = \pi_Y F \pi_Y$ and $F_{Y^\perp} = \pi_{Y^\perp} F \pi_{Y^\perp}$. 

The subspace $Y \subseteq \rr d$ is uniquely determined by $\Lambda$, but the matrix $F$ is not. 
In fact $F_Y$ is uniquely determined, but $F_{Y^\perp}$ can be any matrix such that $Y \subseteq \Ker F_{Y^\perp}$ and $F_{Y^\perp}$ leaves $Y^\perp$ invariant. 

For a symmetric $F \in \M_{d \times d}( \ro )$ we define 
\begin{equation}\label{eq:chichirp}
\chi_F= \begin{pmatrix} I & 0 \\ F & I \end{pmatrix} \in \Sp(d, \ro).
\end{equation}		
The corresponding metaplectic operator is  $\mu(\chi_F)f (x) = e^{\frac{i}{2} \la F x,x \ra} f(x)$.
Note that 
\begin{equation}\label{eq:chiFiso}
\chi_F: Y \times Y^\perp \to \Lambda
\end{equation}		
is an isomorphism. 

We recall the notion of  $\Gamma$-conormal distribution \cite[Definition~5.1]{Cappiello2}. 

\begin{defn}\label{def:Gconormal}
Suppose $Y \subseteq \rr d$ is an $n$-dimensional linear subspace, $0 \leqs n \leqs d$, let $N(Y) = Y \times Y^\perp$, 
and let $V \subseteq T^* \rr d$ be a $d$-dimensional linear subspace such that $N(Y) \oplus V = T^* \rr d$. 
Then $u \in \cS'(\rr d)$ is $\Gamma$-conormal to $Y$ of degree $m\in \ro$, denoted $u \in I^m_\Gamma(\rr d,Y)$, if for any $g \in \cS(\rr d) \setminus 0$ and for any $ k,N \in \mathbb{N}$ we have
\begin{equation}
\label{eq:conormchar}
\begin{aligned}
\left| L_1 \cdots L_k \cT^Y_g u (x,\xi) \right |
& \lesssim \left( 1 + \dist((x,\xi),V) \right)^{m-k} \left( 1 + \dist((x,\xi),N(Y)) \right)^{-N}, \\
& \qquad (x,\xi) \in T^* \rr d, 
\end{aligned}
\end{equation}
where
\begin{equation}\label{eq:TgYdef}
\cT_g^Y u(x,\xi) = e^{-i \la \pi_{Y^\perp} x, \xi\ra} \cT_g u (x, \xi), \quad (x,\xi) \in T^* \rr d, 
\end{equation}
and
$L_j =\la b_j, \nabla_{x,\xi} \ra$ are first order differential operators with $b_j \in N(Y)$, $j=1,\dots,k$. 
\end{defn}

The space $I^m_\Gamma(\rr d,Y)$ is equipped with a topology defined by seminorms of the best constants in \eqref{eq:conormchar}, cf. \cite{Cappiello2}. 

\begin{prop}\label{prop:chirpinvariance}
If $Y \subseteq \rr d$ is a linear subspace, $F \in \M_{d \times d}( \ro )$ is symmetric, $Y \subseteq \Ker F$
and $\chi_F \in \Sp(d,\ro)$ is defined by \eqref{eq:chichirp},
then 
\begin{equation*}
\mu(\chi_F): I_\Gamma^m(\rr d,Y) \to I_\Gamma^m(\rr d,Y) 
\end{equation*}
is a homeomorphism. 
\end{prop}

\begin{proof}
Let $u \in I_\Gamma^m(\rr d,Y)$ and let $g \in \cS(\rr d) \setminus 0$. 
By Lemma \ref{lem:FBIsymplectic}
\begin{equation*}
\cT_{\mu(\chi_F) g} ( \mu(\chi_F) u ) (x,\xi) 
= e^{\frac{i}{2} \la F x,x \ra} \cT_g u (x,\xi-F x). 
\end{equation*}
From \eqref{eq:TgYdef} we obtain
\begin{align*}
\cT_{\mu(\chi_F) g}^Y ( \mu(\chi_F) u ) (x,\xi) 
& = e^{\frac{i}{2} \la F x,x \ra - i \la \pi_{Y^\perp} x, \xi \ra} \cT_g u (x,\xi-F x) \\
& = e^{-\frac{i}{2} \la F x,x \ra - i \la \pi_{Y^\perp} x, \xi - F x \ra} \cT_g u (x,\xi-F x) \\
& = e^{-\frac{i}{2} \la F x,x \ra} \cT_g^Y u (x,\xi-F x). 
\end{align*}

A differential operator of the form $\la a, \nabla_x \ra$ where $a \in Y$, applied to 
$e^{-\frac{i}{2} \la F x,x \ra}$ equals zero, due to the assumption $Y \subseteq \Ker F$. 
Therefore we get from Definition \ref{def:Gconormal}, for any $k,N \in \mathbb{N}$
\begin{equation}\label{eq:conormalchirp}
\begin{aligned}
& \left| L_1 \cdots L_k \cT_{\mu(\chi_F) g}^Y ( \mu(\chi_F) u ) (x,\xi)  \right | \\
& \lesssim \left( 1 + \dist((x,\xi-F x),N(Y^\perp)) \right)^{m-k} \left( 1 + \dist((x,\xi-F x),N(Y)) \right)^{-N}, \\
& \qquad (x,\xi) \in T^* \rr d, 
\end{aligned}
\end{equation}
where 
$L_j = \langle b_j, \nabla_{x,\xi} \rangle$ and $b_j \in N(Y)$, $j=1,\dots,k$.

We have  
\begin{equation*}
\dist^2(\xi-Fx,Y^\perp) = |\pi_{Y} (\xi-Fx) |^2 = |\pi_{Y} \xi |^2
= \dist^2(\xi,Y^\perp). 
\end{equation*}
By means of \eqref{eq:Peetre} we estimate
\begin{align*}
1 + \dist^2((x,\xi-Fx),N(Y^\perp)) 
& = \eabs{(\pi_{Y} x, \pi_{Y^\perp} (\xi-Fx))}^2 \\
& \lesssim \eabs{(\pi_{Y} x, \pi_{Y^\perp} \xi)}^2
\eabs{ \pi_{Y^\perp} x }^2 \\
& \lesssim (1 + \dist^2((x,\xi),N(Y^\perp))) \ (1 + \dist^2((x,\xi),N(Y)))
\end{align*}
and similarly
\begin{align*}
& 1 + \dist^2((x,\xi),N(Y^\perp)) \\
& \lesssim (1 + \dist^2((x,\xi-Fx),N(Y^\perp))) \ (1 + \dist^2((x,\xi),N(Y))). 
\end{align*}
Thus for any $s \in \ro$
\begin{align*}
& (1 + \dist((x,\xi-Fx),N(Y^\perp)))^s \\
& \lesssim (1 + \dist((x,\xi),N(Y^\perp)))^s \ (1 + \dist((x,\xi),N(Y)))^{|s|}, 
\end{align*}
and it follows upon insertion into \eqref{eq:conormalchirp} that we have 
\begin{equation*}
\begin{aligned}
& \left| L_1 \cdots L_k \cT_{\mu(\chi_F) g}^Y ( \mu(\chi_F) u ) (x,\xi)  \right | \\
& \lesssim \left( 1 + \dist((x,\xi),N(Y^\perp)) \right)^{m-k} \left( 1 + \dist((x,\xi),N(Y)) \right)^{-N}, \\
& \qquad (x,\xi) \in T^* \rr d, 
\end{aligned}
\end{equation*}
for any $k,N \in \no$. 

By virtue of Definition \ref{def:Gconormal} 
we have proven that $\mu(\chi_F)$ maps $I_\Gamma^m(\rr d,Y)$ into itself, and 
the continuity is a consequence of the argument. 
The inverse of $\mu(\chi_F)$ is also continuous since $\chi_F^{-1} = \chi_{-F}$. 
\end{proof}

Now we can define $\Gamma$-Lagrangian distributions.

\begin{defn}\label{def:Lagrangiandistribution}
Suppose $\Lambda \subseteq T^* \rr d$ is a Lagrangian defined by a linear subspace $Y \subseteq \rr d$ and a symmetric matrix $F \in \M_{d \times d}( \ro )$ such that $F: Y \to Y$, as in \eqref{eq:Lagrangian}. 
Then $u \in \cS'(\rr d)$ is called a $\Gamma$-Lagrangian distribution with respect to $\Lambda$ of order $m \in \ro$, denoted $u \in I_\Gamma^m(\rr d, \Lambda)$ if $u= \mu(\chi_F) v$ for some $v \in I_\Gamma^m(\rr d,Y)$. 
\end{defn}

\begin{rem}
Note that $\cS(\rr d) \subseteq I_\Gamma^m(\rr d, \Lambda)$ for any Lagrangian $\Lambda \subseteq T^* \rr d$, cf. \cite[Corollary~5.10]{Cappiello2}.
Hence we may calculate modulo Schwartz functions when determining whether a distribution is $\Gamma$-Lagrangian.
\end{rem}

As discussed above the matrix $F$ is not unique in that $F_{Y^\perp}$ may be arbitrary within its stipulated restrictions. 
But since $\chi_F = \chi_{F_Y + F_{Y^\perp}} = \chi_{F_Y} \chi_{F_{Y^\perp}}$
implies 
\begin{equation*}
\mu(\chi_F) = \pm \mu(\chi_{F_Y})  \mu(\chi_{F_{Y^\perp}}), 
\end{equation*}
Definition \ref{def:Lagrangiandistribution} does not depend on $F_{Y^\perp}$, due to Proposition \ref{prop:chirpinvariance}. 

The space $I^m_\Gamma(\rr d,\Lambda)$ is endowed with the topology on $v \in I_\Gamma^m(\rr d,Y)$ referring to the factorization $u= \mu(\chi_F) v$ of $u \in I^m_\Gamma(\rr d,\Lambda)$. 
Again Proposition \ref{prop:chirpinvariance} serves to rid the topology on $I^m_\Gamma(\rr d,\Lambda)$ of dependence on the matrix $F$. 

\begin{rem}
The space $I_\Gamma^m(\rr d, \Lambda)$ reduces to $I_\Gamma^m(\rr d, Y)$ when $\Lambda$ is of the form $\Lambda = Y \times Y^\perp \subseteq T^* \rr d$ for a linear subspace $Y \subseteq \rr d$. 
\end{rem}

\begin{example}
Suppose $1 \leqs n \leqs d-1$, $k=d-n$, $Y= \rr n \times \{0\} \subseteq \rr d$, and 
\begin{equation*}
F = 
\left(
\begin{array}{ll}
A & 0 \\
0 & 0
\end{array}
\right)
\end{equation*}
where $A \in \M_{n \times n}(\ro)$ is symmetric. 
Then 
\begin{equation*}
\Lambda = \{(x_1,0, A x_1, x_2) \in T^* \rr d: x_1 \in \rr n, \ x_2 \in \rr k \}. 
\end{equation*}
By \cite[Lemma~5.4]{Cappiello2} a distribution $u\in I_\Gamma^m(\rr d,\Lambda)$ is of the form
\begin{equation*}
u(x_1,x_2) = \int_{\rr k}  e^{i \left( \frac{1}{2} \la x_1 ,A x_1 \ra + \la x_2,\theta \ra \right)} a(x_1,\theta)\,\dd\theta
\end{equation*}
for $a \in \Gamma^m(\rr d )$.
\end{example}

As observed before Proposition \ref{prop:contSobolevShubin},
$\mu(\chi)$ is a homeomorphism on $Q^s (\rr d)$ for any $\chi \in \Sp(d,\ro)$ and any $s \in \ro$. 
From the estimates \eqref{eq:conormchar} we obtain therefore for any $\ep>0$
\begin{equation*}
I_\Gamma ^m(\rr d,\Lambda) \subseteq Q^{-\left(m+\frac{d}{2}+\ep\right)} (\rr d).
\end{equation*}
Microlocally, $\Gamma$-Lagrangian distributions are however usually more regular than generic elements of $Q^{-\left(m+\frac{d}{2}+\ep\right)}(\rr d)$. 

Combining \cite[Proposition~5.17]{Cappiello2}, \cite[Eq.~(2.18)]{Rodino1} 
and \eqref{eq:chiFiso} gives

\begin{prop}\label{prop:WFLagrangian}
If $u\in I^m_\Gamma(\rr d,\Lambda)$ then $\WF(u) \subseteq \Lambda$.
\end{prop}

\begin{lem}\label{lem:conormfactor}
Let $0 \leqs n \leqs d$, and
suppose $U = [M_1 \ M_2] \in \On(d)$ with $M_1 \in \M_{d \times n}(\ro)$ and 
$M_2 \in \M_{d \times (d-n)}(\ro)$ and $Y = \Ker M_2^t \subseteq \rr d$. 
Define 
\begin{equation}\label{eq:chiUdef}
\chi^U = 
\left(
  \begin{array}{cc}
  U & 0 \\
  0 & U
  \end{array}
\right) \in \Sp(d,\ro)
\end{equation}
and
\begin{equation}\label{eq:J2invdef}
\J_2^{-1} =
\left(
\begin{array}{cccc}
I_n & 0 & 0 & 0 \\
0 & 0 & 0 & -I_{d-n} \\
0 & 0 & I_n & 0 \\
0 & I_{d-n} & 0 & 0 
\end{array}
\right) \in \Sp(d,\ro).  
\end{equation}
Then 
\begin{equation*}
\chi^U  \J_2^{-1}: \rr d \times \{ 0 \} \to Y \times Y^\perp
\end{equation*}
is an isomorphism. 
\end{lem}

\begin{proof}
We have $n = \dim Y$. 
The assumptions give
\begin{equation*}
\chi^U  \J_2^{-1} 
= \left(
\begin{array}{cccc}
M_1 & 0 & 0 & -M_2 \\
0 & M_2 & M_1 & 0 \\
\end{array}
\right) \in \Sp(d,\ro). 
\end{equation*}
Denoting $x=(x_1,x_2) \in \rr d$ with $x_1 \in \rr n$ and $x_2 \in \rr {d-n}$ we have
\begin{equation*}
\chi^U  \J_2^{-1} (x,0) 
= (M_1 x_1, M_2 x_2), \quad x \in \rr d, 
\end{equation*}
which proves the claim since $Y = \Ran M_1$ and $Y^\perp = \Ran M_2$.
\end{proof}

\begin{lem}\label{lem:sympiso}
If $\chi \in \Sp(d,\ro)$ preserves $\rr d \times \{0\}$ then $\mu(\chi)$ is a homeomorphism on $\Gamma^m(\rr d)$. 
\end{lem}

\begin{proof}
Using the block matrix notation \eqref{eq:symplecticABCD}, the properties \eqref{eq:sympmatrix1}, \eqref{eq:sympmatrix2} and \eqref{eq:sympinvABCD}, 
the assumption entails
\begin{equation*}
\chi
= \left( 
\begin{array}{cc}
A & B \\
0 & A^{-t}
\end{array}
\right)
= -\J \left( 
\begin{array}{cc}
I & 0 \\
-BA^t & I
\end{array}
\right)
\J
\left( 
\begin{array}{cc}
A & 0 \\
0 & A^{-t}
\end{array}
\right).
\end{equation*}
Note that $BA^t$ is symmetric and
\begin{equation*}
T f(x)  := 
\mu
\left( 
\begin{array}{cc}
A & 0 \\
0 & A^{-t}
\end{array}
\right) f(x)
= |A|^{-1/2} f(A^{-1} x) = |A|^{-1/2} (A^{-1})^* f (x)
\end{equation*}
for $f \in \cS(\rr d)$. 
Combined with \eqref{eq:metaplecticFourier} and the notation \eqref{eq:chichirp} this gives
\begin{equation*}
\mu(\chi)=\pm \cF^{-1} \mu(\chi_{-BA^t}) \cF 
T. 
\end{equation*}

Clearly $T$ is a homeomorphism on $\Gamma^m(\rr d)$. 
By \cite[Corollary~5.5 and Proposition~5.12]{Cappiello2} $\cF: \Gamma^m(\rr d) \to I_\Gamma^m(\rr d, \{0\})$ is a homeomorphism. 
The claim is hence a consequence of the fact that $\mu(\chi_{-BA^t})$ is a homeomorphism on $I_\Gamma^m(\rr d, \{0\})$, 
which is granted by Proposition \ref{prop:chirpinvariance}. 
\end{proof}

Next we observe that pseudodifferential operators act well on $\Gamma$-Lagrangian distributions.
This generalizes \cite[Proposition~5.19]{Cappiello2}.  

\begin{lem}\label{lem:pseudomap}
Let $a \in \Gamma^{m'}(\rr {2d})$ and suppose $\Lambda \subseteq T^* \rr d$ is a Lagrangian. 
Then
\begin{equation*}
a^w(x,D):I^m_\Gamma(\rr d,\Lambda) \to I^{m+m'}_\Gamma(\rr d,\Lambda)
\end{equation*}
is continuous. 
\end{lem}

\begin{proof}
In \cite[Proposition~5.19]{Cappiello2} the continuity
\begin{equation}\label{eq:contconormal}
a^w(x,D):I^m_\Gamma(\rr d,Y) \to I^{m+m'}_\Gamma(\rr d,Y)
\end{equation}
is proved. 
Suppose $u \in I^m_\Gamma(\rr d,\Lambda)$, that is $u=\mu(\chi_F) v$ where $v \in I^m_\Gamma(\rr d,Y)$ and where $F \in \M_{d \times d}(\ro)$ and $Y \subseteq \rr d$ are associated to $\Lambda$ as in \eqref{eq:Lagrangian}. 
We obtain using \eqref{eq:metaplecticoperator}
\begin{align*}
a^w(x,D)u 
& = \mu(\chi_{F}) \mu(\chi_{F})^{-1 }a^w(x,D) \mu(\chi_{F}) v \\
& = \mu(\chi_{F}) (a \circ \chi_F)^w(x,D) v 
\end{align*}
which proves the result since $(a \circ \chi_F)^w(x,D) v \in  I^{m+m'}_\Gamma(\rr d,Y)$ by \eqref{eq:contconormal}. 
The continuity claim is a consequence of the continuity \eqref{eq:contconormal} and the definition of the topology on $I^m_\Gamma(\rr d,\Lambda)$.
\end{proof}

With the help of Lemma \ref{lem:pseudomap} we can prove a continuity result for FIOs acting on $\Gamma$-Lagrangian distributions.  
Note that $\chi \Lambda \subseteq T^* \rr d$ is Lagrangian provided $\Lambda \subseteq T^* \rr d$ is Lagrangian and $\chi \in \Sp(d, \ro )$. 

\begin{thm}
\label{thm:FIOonLag}
Suppose $\chi \in \Sp(d, \ro )$, $\cK \in \cI^{m'}(\chi)$ and let $\Lambda \subseteq T^* \rr d$ be a Lagrangian. 
Then
\begin{equation*}
\cK: I^m_\Gamma(\rr d,\Lambda) \to  I^{m+m'}_\Gamma (\rr d,\chi \Lambda)
\end{equation*}
is continuous. 
\end{thm}
\begin{proof}
By Theorem \ref{thm:repFIO}, $\cK = b^w(x,D) \mu(\chi)$ for $b \in \Gamma^{m'}(\rr {2d})$. 
Appealing to Lemma \ref{lem:pseudomap} it therefore suffices to show that 
\begin{equation*}
\mu(\chi):I^m_\Gamma(\rr d,\Lambda) \to I^m_\Gamma (\rr d,\chi \Lambda)
\end{equation*}
is continuous. 

Suppose $\Lambda \subseteq T^* \rr d$ is parametrized by $Y \subseteq \rr d$ and $F \in \M_{d \times d}(\ro)$ as in \eqref{eq:Lagrangian}, 
and likewise that the Lagrangian $\chi \Lambda \subseteq T^* \rr d$ is parametrized by $Y' \subseteq \rr d$ and $F' \in \M_{d \times d}(\ro)$. 
Let $u \in I^m_\Gamma(\rr d,\Lambda)$ so that $u = \mu(\chi_F) v$ with $v \in I^m_\Gamma(\rr d,Y)$.

We need to show 
\begin{align}\label{eq:redstep}
\mu(\chi)u=\mu(\chi) \mu(\chi_F) v=\mu(\chi_{F'}) v'
\end{align}
for some $v' \in I^m_\Gamma(\rr d, Y')$. 

Set $n = \dim Y$. 
By \cite[Proposition~5.9]{Cappiello2} 
we have $v \in I^m_\Gamma(\rr d,Y)$ if and only if there exists $a \in \Gamma^m(\rr d)$
and $U = [M_1 \ M_2] \in \On(d)$, with $M_1 \in \M_{d \times n}(\ro)$,  
$M_2 \in \M_{d \times (d-n)}(\ro)$ and $Y = \Ker M_2^t$, such that $v = U^{t *} \cF_2^{-1} a$. 
Since $U^{t *} = \mu(\chi^U)$ and $\cF_2^{-1} = \mu(\J_2^{-1})$ using the notation \eqref{eq:chiUdef} and \eqref{eq:J2invdef}, 
we may write $v = \pm \mu(\chi^U \J_2^{-1}) a$. 

Thus \eqref{eq:redstep} may be written
\begin{equation*}
\pm \mu(\chi \, \chi_F \, \chi^U \J_2^{-1}) a 
= \mu(\chi_{F'}) v '. 
\end{equation*}

Again by \cite[Proposition~5.9]{Cappiello2}, the claim $v' \in I^m_\Gamma(\rr d, Y')$ can be proved by showing
$v' = \mu(\chi^V \J_2^{-1}) b$ where $b \in \Gamma^m(\rr d)$, 
$V = [N_1 \ N_2] \in \On(d)$, with $N_1 \in \M_{d \times k}(\ro)$ and 
$N_2 \in \M_{d \times (d-k)}(\ro)$ such that $Y' = \Ker N_2^t$ and $k = \dim Y'$. 
 
With these terms we must show
\begin{equation*}
b = \mu( \J_2 \, \chi^{V^t} \, \chi_{F'}^{-1} \, \chi \, \chi_F \, \chi^U \, \J_2^{-1}) a \in \Gamma^m(\rr d) 
\end{equation*}
and also the continuity of $a \mapsto b$ on $\Gamma^m(\rr d)$ (cf. \cite{Cappiello2}). 

Set 
\begin{equation*}
\chi_0 = \J_2 \, \chi^{V^t} \, \chi_{F'}^{-1} \, \chi \, \chi_F \, \chi^U \, \J_2^{-1} \in \Sp(d,\ro)
\end{equation*}
so that $b = \mu(\chi_0) a$.

From Lemma \ref{lem:conormfactor}, and by definition of $\Lambda$ and $\chi \Lambda$, and \eqref{eq:chiFiso}, we obtain the following sequence of isomorphisms concerning the symplectic matrices at hand. 
\begin{equation*}
\rr d \times \{0\} \stackrel{\chi^U \J_2^{-1}} {\longrightarrow} Y\times Y^\perp \stackrel{\chi_F}{\longrightarrow} \Lambda \stackrel{\chi}{\longrightarrow} \chi \Lambda \stackrel{\chi_{F'}^{-1}}{\longrightarrow} Y' \times Y'^\perp \stackrel{\J_2 \chi^{V^t}}{\longrightarrow} \rr d \times \{0\}.
\end{equation*}
Hence $\chi_0$ restricts to an isomorphism on $\rr d \times\{0\}$. 
The claim is thus a consequence of Lemma \ref{lem:sympiso}. 
\end{proof}

Lemma \ref{lem:conormfactor}, 
\eqref{eq:chiFiso} and the proof of Theorem \ref{thm:FIOonLag} give the following characterization of $\Gamma$-Lagrangian distributions. 

\begin{cor}\label{cor:Lagmeta}
A distribution $u\in\cS'(\rr{d})$ satisfies $u\in I_\Gamma^m(\rr{d},\Lambda)$ if and only if there exist $\chi \in \Sp(d,\ro)$ that maps $\chi: \rr d \times\{0\} \to \Lambda$ isomorphically, and $a \in \Gamma^m(\rr d)$ such that $u= \mu (\chi) a$. 
\end{cor}

\begin{rem}\label{rem:Lagmeta}
Given a Lagrangian $\Lambda \subseteq T^* \rr d$, 
the existence of $\chi \in \Sp(d,\ro)$ with the stipulated property is a consequence of 
Lemma \ref{lem:conormfactor} and \eqref{eq:chiFiso}. 
By Lemma \ref{lem:sympiso}, the equivalent statement in Corollary \ref{cor:Lagmeta} can be reformulated 
as follows. 
\emph{For all $\chi \in \Sp(d,\ro)$ that maps $\chi: \rr d \times\{0\} \to \Lambda$ isomorphically there exists $a \in \Gamma^m(\rr d)$ such that $u= \mu (\chi) a$.}
\end{rem}

Finally we prove a time-frequency characterization of $\Gamma$-Lagrangian distributions similar to that of conormal distributions, see Definition \ref{def:Gconormal}. 
Without loss of generality we may assume $Y^\perp \subseteq \Ker F$. 

\begin{prop}\label{prop:Lagrangianchar}
Let $\Lambda \subseteq T^* \rr d$ be a Lagrangian and let $V \subseteq T^* \rr d$ be a subspace transversal to $\Lambda$. 
Suppose $\Lambda$ is parametrized by $Y \subseteq \rr d$ and $F \in \M_{d \times d}(\ro)$ as in \eqref{eq:Lagrangian}.  
A distribution $u \in \cS'(\rr d)$ satisfies $u \in I_\Gamma^m(\rr d,\Lambda)$ if and only if
for any $g \in \cS(\rr d)\setminus 0$ and for any $k,N \in \mathbb{N}$ we have
\begin{equation}
\label{eq:lagrangianchar}
\begin{aligned}
\left| L_1 \cdots L_k \cT^\Lambda_g u (x,\xi) \right |
& \lesssim \left( 1 + \dist((x,\xi),V) \right)^{m-k} \left( 1 + \dist((x,\xi),\Lambda) \right)^{-N}, \\
& \qquad (x,\xi) \in T^* \rr d, 
\end{aligned}
\end{equation}
where
\begin{equation}\label{eq:TgLdef}
\cT_g^\Lambda u(x,\xi) = e^{-i \left( \la \pi_{Y^\perp} x, \xi \ra + \frac{1}{2} \la x, Fx\ra \right)} \cT_g u (x, \xi), \quad (x,\xi) \in T^* \rr d, 
\end{equation}
and $L_j =\la b_j, \nabla_{x,\xi} \ra$ are first order differential operators with $b_j \in \Lambda$, $j=1,\dots,k$. 
\end{prop}

\begin{proof}
Note that \eqref{eq:lagrangianchar} and \eqref{eq:TgLdef} reduce to \eqref{eq:conormchar} and \eqref{eq:TgYdef}, respectively, when $\Lambda = Y \times Y^\perp$.

We have $u \in I_\Gamma^m(\rr d,\Lambda)$ if and only if $u = \mu(\chi_F) v$ where $v \in I_\Gamma^m(\rr d,Y)$. 
By Proposition \ref{prop:chirpinvariance} we may assume $Y^\perp \subseteq \Ker F$.
Let $g \in \cS(\rr d)\setminus 0$ and set $h = \mu(\chi_F)^{-1} g \in \cS(\rr d)\setminus 0$. 
Lemma \ref{lem:FBIsymplectic} and \eqref{eq:TgYdef} give
\begin{align*}
\cT_g^\Lambda u (x,\xi) 
& = e^{-i \la \pi_{Y^\perp} x, \xi \ra} \cT_h v(x,-Fx + \xi) \\
& = \cT_h^Y v(x,-Fx + \xi). 
\end{align*}

A differential operator $L =\la Z, \nabla_{x,\xi} \ra$ with $Z \in \Lambda = \chi_F(Y \times Y^\perp)$ is of the form
\begin{equation*}
\la a, \nabla_x + F \nabla_\xi \ra + \la b, \nabla_\xi \ra
\end{equation*}
where $a \in Y$ and $b \in Y^\perp$. 
This operator acts as 
\begin{align*}
& \left( \la a, \nabla_x + F \nabla_\xi \ra + \la b, \nabla_\xi \ra \right) ( \cT_g^\Lambda u (x,\xi) ) \\
& = \left( \left( \la a, \nabla_1 - F \nabla_2 + F \nabla_2 \ra  + \la b, \nabla_2 \ra \right) \cT_h^Y v \right) (x,-Fx + \xi) \\
& = \left( \la (a,b), \nabla_{1,2} \ra \cT_h^Y v \right) (x,-Fx + \xi). 
\end{align*}
The claim is now a consequence of Definition \ref{def:Gconormal}. 
To wit, 
\begin{align*}
\dist((x,-Fx+\xi), N(Y)) & = \dist( \chi_{-F}(x,\xi), Y \times Y^\perp) 
\asymp \dist((x,\xi), \Lambda), \\
\dist((x,-Fx+\xi), N(Y^\perp)) & \asymp \dist( (x,\xi), \chi_{F}( Y^\perp \times Y) ), 
\end{align*}
and $\chi_{F}( Y^\perp \times Y) \subseteq T^* \rr d$ is transversal to $\Lambda$. 
\end{proof}

\section{Kernels of FIOs and $\Gamma$-Lagrangian distributions}\label{sec:kernelLagrangian}

In this section we prove that the kernels of FIOs associated to $\chi \in \Sp(d,\ro)$ are the $\Gamma$-Lagrangian distributions associated with the twisted graph Lagrangian $\Lambda_\chi'$. 
\begin{lem}\label{lem:metaplectictensor}
If $\chi \in \Sp(d,\ro)$ then there exists $\theta \in \ro$ such that 
\begin{equation*}
\mu(\chi) \otimes \mathrm{id} = e^{i \theta} \mu(\chi_2)
\end{equation*}
where $\chi_2 \in \Sp(2d,\ro)$ is defined by
\begin{align*}
\chi_2(x,\xi) & = \left( \chi(x_1,\xi_1)_1, x_2, \chi(x_1,\xi_1)_2, \xi_2 \right), \\
& \qquad \qquad x=(x_1,x_2) \in \rr {2d}, \quad \xi=(\xi_1,\xi_2) \in \rr {2d}. 
\end{align*}
\end{lem}

\begin{proof}
Let $f,g,h,q \in \cS(\rr d)$. 
From \eqref{eq:wignerweyl}, \eqref{eq:wignerdistribution} and \eqref{eq:metaplecticoperator} we obtain 
\begin{align*}
W(\mu(\chi) h \otimes q, \mu(\chi) f \otimes g)(x,\xi)
& = W(\mu(\chi) h , \mu(\chi) f)(x_1,\xi_1) \, W(q,g)(x_2,\xi_2) \\
& = W(h , f)(\chi^{-1} (x_1,\xi_1)) \, W(q,g)(x_2,\xi_2).
\end{align*}
Again using \eqref{eq:wignerweyl} this gives for $a \in \cS(\rr {4d})$
\begin{align*}
& \left( (\mu(\chi)^{-1} \otimes \mathrm{id} ) a^w(x,D) (\mu(\chi) \otimes \mathrm{id})(f \otimes g), h \otimes q \right) \\
& \qquad \qquad = \left( a^w(x,D) (\mu(\chi) f \otimes g), \mu(\chi) h \otimes q \right) \\
& \qquad \qquad = (2 \pi)^{-d} (a, W(\mu(\chi) h \otimes q, \mu(\chi) f \otimes g) ) \\
& \qquad \qquad = (2 \pi)^{-d} (b, W(h \otimes q, f \otimes g) ) \\
& \qquad \qquad = (b^w(x,D) (f \otimes g) , h \otimes q) 
\end{align*}
where 
\begin{align*}
b(x,\xi) & = a \left( \chi(x_1,\xi_1)_1, x_2, \chi(x_1,\xi_1)_2, \xi_2 \right) \\
& = a(\chi_2(x,\xi)). 
\end{align*}

Appealing to \cite[Theorem~51.6]{Treves1} we have thus shown 
\begin{equation}\label{eq:metaplectictensor}
(\mu(\chi)^{-1} \otimes \mathrm{id} ) \, a^w(x,D) \, (\mu(\chi) \otimes \mathrm{id})
= (a \circ \chi_2)^w(x,D).
\end{equation}
The following calculation for $(x,\xi), (y,\eta) \in T^* \rr {2d}$ shows that $\chi_2 \in \Sp(2d,\ro)$. 
\begin{align*}
& \sigma \left( \chi_2(x,\xi), \chi_2(y,\eta) \right) \\
& = \sigma \left( (\chi(x_1,\xi_1)_1, x_2, \chi(x_1,\xi_1)_2, \xi_2) , ( \chi(y_1,\eta_1)_1, y_2, \chi(y_1,\eta_1)_2, \eta_2 )\right) \\
& = \la \chi(x_1,\xi_1)_2, \chi(y_1,\eta_1)_1 \ra + \la \xi_2, y_2\ra - \la \chi(x_1,\xi_1)_1, \chi(y_1,\eta_1)_2 \ra - \la x_2, \eta_2 \ra \\
& = \sigma(\chi(x_1,\xi_1), \chi(y_1,\eta_1)) + \sigma(x_2,\xi_2, y_2, \eta_2) \\
& = \sigma(x_1,\xi_1, y_1,\eta_1) + \sigma(x_2,\xi_2, y_2, \eta_2) \\
& = \la y_1, \xi_1 \ra - \la x_1, \eta_1 \ra + \la y_2, \xi_2 \ra - \la x_2, \eta_2 \ra \\
& = \sigma( (x,\xi), (y, \eta) ).
\end{align*}

Combining \eqref{eq:metaplectictensor} with \eqref{eq:wignerweyl} and \eqref{eq:metaplecticoperator} we get 
\begin{equation*}
W( (\mu(\chi) \otimes \mathrm{id}) g, (\mu(\chi) \otimes \mathrm{id}) f) 
= W( \mu(\chi_2) g, \mu(\chi_2) f)
\end{equation*}
for all $f,g \in \cS(\rr {2d})$ which implies 
$\mu(\chi) \otimes \mathrm{id} = e^{i \theta} \mu(\chi_2)$ for some $\theta \in \ro$. 
\end{proof}

\begin{thm}\label{thm:FIOkernelLagr}
If $\chi \in \Sp(d,\ro)$ then 
\begin{equation*}
K^m(\chi) = I_\Gamma^m(\rr {2d}, \Lambda_\chi'). 
\end{equation*}
\end{thm}

\begin{proof}
First we prove $K^m(\chi) \subseteq I_\Gamma^m(\rr {2d}, \Lambda_\chi')$. 

Let $K_{\fy,a} \in K^m(\chi)$. 
By Theorem \ref{thm:repFIO} we have $\cK_{\fy,a} = \mu(\chi)b^w(x,D)$ for some $b \in \Gamma^m(\rr {2d})$. 
Define
\begin{equation*}
\chi_\Delta = 
\left(
\begin{array}{cccc}
I_d & 0 & 0 & 0 \\
I_d & 0 & 0 & I_d \\
0 & I_d & I_d & 0 \\
0 & -I_d & 0 & 0
\end{array}
\right) \in \M_{4d \times 4d} (\ro).
\end{equation*}
Then $\chi_\Delta \in \Sp(2d,\ro)$ and 
\begin{equation*}
\chi_\Delta: \rr {2d} \times \{ 0 \} \to \Delta \times \Delta^\perp
\end{equation*}
isomorphically, cf. \eqref{eq:diagonal}. 
The kernel of $b^w(x,D)$ is denoted $K_b$ (cf. \eqref{eq:schwartzkernelpseudo}), 
and $K_b \in I_{\Gamma}^m(\rr {2d},\Delta)$ (see \cite[Example~5.2]{Cappiello2}).
By Corollary \ref{cor:Lagmeta} and Remark \ref{rem:Lagmeta} we have $K_b = \mu(\chi_\Delta) b_1$ for some $b_1 \in \Gamma^m(\rr {2d})$. 

This gives for $f,g \in \cS(\rr d)$
\begin{align*}
(K_{\fy,a}, g\otimes \overline{f})
&= (\cK_{\fy,a} f,g) \\
&= (\mu(\chi)b^w(x,D) f,g) \\
&=(b^w(x,D) f,\mu(\chi)^{-1} g) \\
&=(K_b,\mu(\chi)^{-1} g\otimes \overline{f}) \\
&=(\mu(\chi_\Delta) b_1, \mu(\chi)^{-1} g \otimes \overline{f}) \\
&=((\mu(\chi) \otimes \mathrm{id}) \mu(\chi_\Delta) b_1, g\otimes\overline{f})
\end{align*}
and it follows
\begin{equation}\label{eq:fiokernel1}
K_{\fy,a} = (\mu(\chi) \otimes \mathrm{id}) \mu(\chi_\Delta) b_1. 
\end{equation}

By Lemma \ref{lem:metaplectictensor} we have 
\begin{equation}\label{eq:metaplectictensor2}
\mu(\chi) \otimes \mathrm{id} = e^{i \theta} \mu(\chi_2)
\end{equation}
where $\theta \in \ro$ and $\chi_2 \in \Sp(2d,\ro)$ is defined by
\begin{align*}
\chi_2(x,\xi) & = \left( \chi(x_1,\xi_1)_1, x_2, \chi(x_1,\xi_1)_2, \xi_2 \right), \\
& \qquad \qquad x=(x_1,x_2) \in \rr {2d}, \quad \xi=(\xi_1,\xi_2) \in \rr {2d}. 
\end{align*}
Insertion of \eqref{eq:metaplectictensor2} into \eqref{eq:fiokernel1} yields
\begin{equation}\label{eq:kernelFIOLagrangian}
K_{\fy,a} = \pm e^{i \theta} \mu(\chi_2 \chi_\Delta) b_1. 
\end{equation}

For $(x,\xi) \in \rr {2d}$ we obtain 
\begin{align*}
\chi_2 \chi_\Delta(x,\xi,0,0) 
& = \chi_2( x,x,\xi,-\xi ) \\
& = \left( \chi(x,\xi)_1, x, \chi(x,\xi)_2, -\xi \right)
\end{align*}
and it follows that 
\begin{equation*}
\chi_2 \chi_\Delta: \rr {2d} \times \{ 0 \} \to \Lambda_\chi'
\end{equation*}
isomorphically. 
Again appealing to Corollary \ref{cor:Lagmeta} we may conclude from 
\eqref{eq:kernelFIOLagrangian} that 
$K_{\fy,a} \in I_\Gamma^m(\rr {2d}, \Lambda_\chi')$. This proves $K^m(\chi) \subseteq I_\Gamma^m(\rr {2d}, \Lambda_\chi')$. 

It remains to show the opposite inclusion $I_\Gamma^m(\rr {2d}, \Lambda_\chi') \subseteq K^m(\chi)$. 
Let $K \in I_\Gamma^m(\rr {2d}, \Lambda_\chi')$, and denote by $\cK$ the operator with kernel $K$. 
By Corollary \ref{cor:Lagmeta} and Remark \ref{rem:Lagmeta} we have $K = \mu(\chi_2 \chi_\Delta) b$ for some $b \in \Gamma^m(\rr {2d})$, 
and $K_1 = \mu(\chi_\Delta) b \in I_{\Gamma}^m(\rr {2d},\Delta)$. 
If we denote by $\cK_1$ the operator with kernel $K_1$ then  
\cite[Example~5.2]{Cappiello2} shows that $\cK_1 = a^w(x,D)$ with $a \in \Gamma^m(\rr{2d})$. 

Using Lemma \ref{lem:metaplectictensor} we obtain for $f,g \in \cS(\rr d)$
\begin{align*}
(\cK f, g)
& = (\mu(\chi_2 \chi_\Delta) b, g \otimes \overline{f} ) \\
& = \pm e^{-i \theta} ((\mu(\chi) \otimes \mathrm{id}) K_1, g \otimes \overline{f} ) \\
& = \pm e^{-i \theta}  ( K_1, \mu(\chi) ^{-1} g \otimes \overline{f} ) \\
& = \pm e^{-i \theta}  ( a^w(x,D) f, \mu(\chi) ^{-1} g) \\
& = \pm e^{-i \theta}  ( \mu(\chi) a^w(x,D) f, g)
\end{align*}
and it follows that $\cK = \mu(\chi) a^w(x,D)$ with $a \in \Gamma^m(\rr{2d})$. 
By Theorem \ref{thm:repFIO} 
this means that $K \in K^m(\chi)$
which proves $I_\Gamma^m(\rr {2d}, \Lambda_\chi') \subseteq K^m(\chi)$. 
\end{proof}

\section*{Acknowledgements}

We are grateful to professor Fabio Nicola for helpful discussions. R. Schulz gratefully acknowledges support of the project ``Fourier Integral Operators, symplectic geometry and analysis on noncompact manifolds'' received by the University of Turin in form of an ``I@Unito'' fellowship as well as institutional support by the university of Hannover.


\end{document}